%% file: main.tex
\documentclass[11pt, a4paper]{article}
\usepackage{amsmath}
\usepackage{mathtools}
\usepackage{amssymb}
\usepackage{graphicx}
\usepackage{algorithm}
\usepackage{algpseudocode}
\usepackage{amsthm}
\usepackage{hyperref}
\usepackage{cleveref}
\usepackage{graphicx}
\usepackage{geometry}
\usepackage{wrapfig}
\usepackage{booktabs}
\usepackage{siunitx}
\usepackage{multirow}

\usepackage{caption}
\usepackage{xcolor}
\usepackage{changepage}
\usepackage{placeins}

\usepackage{physics}
\usepackage{amsmath}
\usepackage{tikz}
\usepackage{mathdots}
\usepackage{yhmath}
\usepackage{cancel}
\usepackage{color}
\usepackage{siunitx}
\usepackage{array}
\usepackage{multirow}
\usepackage{amssymb}
\usepackage{gensymb}
\usepackage{tabularx}
\usepackage{extarrows}
\usepackage{booktabs}
\usetikzlibrary{fadings}
\usetikzlibrary{patterns}
\usetikzlibrary{shadows.blur}
\usetikzlibrary{shapes}
\usepackage{longtable}

\usetikzlibrary{decorations.pathmorphing,arrows.meta,calc,positioning}

\usetikzlibrary{patterns,snakes}
\usetikzlibrary{arrows.meta} 

\usepackage[square,numbers]{natbib}

\theoremstyle{plain}
\newtheorem*{theorem*}{Theorem}
\newtheorem{theorem}{Theorem}[section]

\theoremstyle{definition}

\theoremstyle{remark}
\newtheorem{remark}[theorem]{Remark}

\usepackage{makecell}
\usepackage{tabularray}

\colorlet{xcol}{blue!70!black}
\colorlet{darkblue}{blue!40!black}
\colorlet{myred}{red!65!black}
\tikzstyle{mydashed}=[xcol,dashed,line width=0.25,dash pattern=on 2.2pt off 2.2pt]
\tikzstyle{axis}=[->,thick] 
\tikzstyle{ell}=[{Latex[length=3.3,width=2.2]}-{Latex[length=3.3,width=2.2]},line width=0.3]
\tikzstyle{dx}=[-{Latex[length=3.3,width=2.2]},darkblue,line width=0.3]
\tikzstyle{ground}=[preaction={fill,top color=black!10,bottom color=black!5,shading angle=20},
                    fill,pattern=north east lines,draw=none,minimum width=0.3,minimum height=0.6]
\tikzstyle{mass}=[line width=0.6,red!30!black,fill=red!40!black!10,rounded corners=1,
                  top color=red!40!black!20,bottom color=red!40!black!10,shading angle=20]
\tikzstyle{spring}=[line width=0.8,blue!7!black!80,snake=coil,segment amplitude=5,segment length=5,line cap=round]
\tikzset{>=latex} 
\tikzstyle{force}=[->,myred,very thick,line cap=round]

\title{SIPHy: Sparse identification of port-Hamiltonian systems from noisy data}
\author{

  Håkon Noren Myhr\textsuperscript{1} \and
  Sølve Eidnes \textsuperscript{2} \and
  J. Nathan Kutz \textsuperscript{3} 
}
\date{\today}

\begin{document}

\maketitle

\footnotetext[2]{\scriptsize SINTEF Digital, Oslo, Norway.}
\footnotetext[3]{\scriptsize Autodesk Research, 6 Agar Street, London, United Kingdom.}
\footnotetext[1]{\scriptsize Norwegian University of Science and Technology, Trondheim, Norway (\href{mailto:hakon.noren@ntnu.no}{hakon.noren@ntnu.no}).}

\begin{abstract}
We propose \textit{sparse identification of port-Hamiltonian systems} (SIPHy), a method for structure-preserving symbolic regression from noisy trajectory observations. The method applies to port-Hamiltonian systems, which provide a general framework for describing dynamical systems in terms of energy exchange, dissipation and control. Our algorithm can jointly identify the Hamiltonian as well as the dissipation and input matrices. Furthermore, we introduce \textit{Hamiltonian flow splines} to better approximate derivatives of trajectory data corrupted by noise or with missing time points, a major challenge of system identification for differential equations. This method assembles flows of piecewise polynomial Hamiltonians to produce a smooth, differentiable trajectory necessary for sparse regression. Combining flow splines with SIPHy yields sparse port-Hamiltonian models from noisy and incomplete trajectory data. Since the Hamiltonian, dissipation, and input terms are identified separately, the learned model can also be evaluated after modifying the control input or dissipation, which we demonstrate on a controlled mass–spring system.
\end{abstract}

\begin{figure}[!t]
\centering
\begin{tikzpicture}
\node(img2) at (0,0){\includegraphics[width=0.98\textwidth]{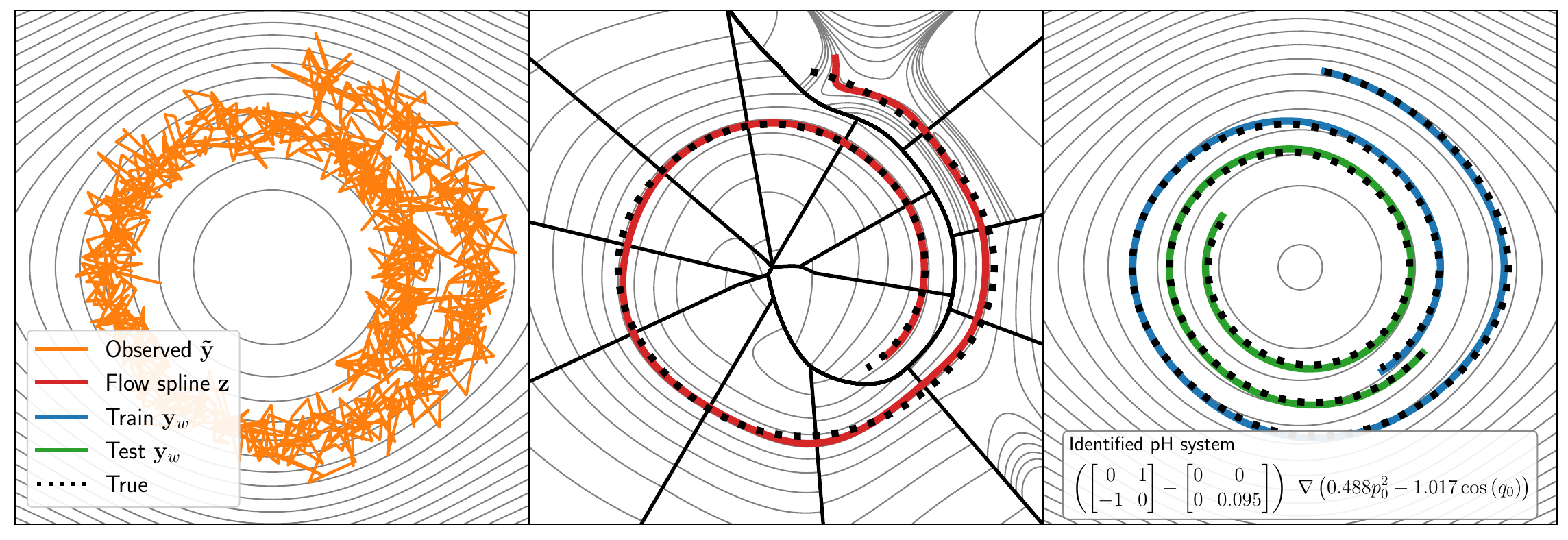}};
\node(img1) at (0,4.75) {\includegraphics[width=0.98\textwidth]{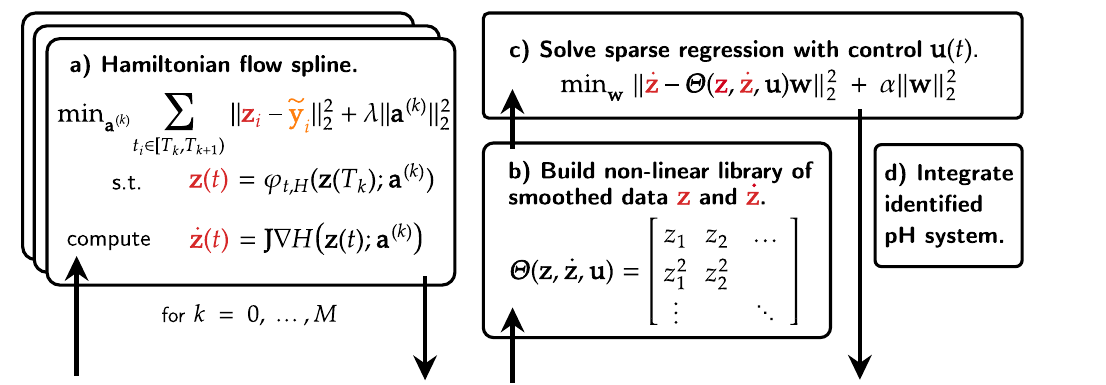}};
\end{tikzpicture}\vspace{-1em}
        \caption{\small \textbf{System identification with Hamiltonian flow splines and SIPHy.} \textbf{(a)} Fit a Hamiltonian flow spline $\mathbf{z}(t)$ (lower, middle) to corrupted trajectory data $\tilde{\mathbf{y}}(t_i)$ from the damped pendulum (lower left). \textbf{(b)} Build a nonlinear feature library $\boldsymbol{\Theta}$ from the preprocessed trajectory and control $\mathbf{u}(t)$. \textbf{(c)} Solve sparse regression to identify the port-Hamiltonian (pH) parameters $\mathbf{w}$. \textbf{(d)} Integrate the identified pH system (lower right) from the initial value corresponding to the observed trajectory ($\mathbf{y}_w$ train) and a new initial value ($\mathbf{y}_w$ test). The level sets show energy levels for the true and learned Hamiltonians.}
        \label{fig:overview}
    \end{figure}

\section{Introduction}

Physics-informed machine learning aims to combine data-driven methods with prior physical knowledge about the underlying dynamics to be modelled. A class of methods within this framework allows models to be \textit{hard-constrained} to inherit physical principles or geometric properties such as energy preservation, symmetries or conservation laws of the dynamical system \cite{greydanusHamiltonianNeuralNetworks2019,offen2024learning,NEURIPS2021_8b519f19}. Even though these approaches enable accurate forecasts of dynamical systems, many rely on neural networks, which often provide little or no interpretability of the discovered systems. Symbolic regression \cite{brunton_discovering_2016,petersen2019deep,cranmer2023interpretable}, on the other hand, identifies a parsimonious mathematical expression of the dynamical system from data~\cite{kutz2022parsimony}. Such algorithms provide forecasts as well as an interpretable representation, facilitating scientific discovery. However, few symbolic regression algorithms allow physics constraints to be enforced.  While the Hamiltonian formalism describes dynamical systems in terms of energy conservation, port-Hamiltonian systems \cite{schaftPortHamiltonianSystemsTheory2014} are a generalized framework for describing multi-physical systems through energy exchange, dissipation and interconnection. Port-Hamiltonian systems are thus naturally suited for data-driven modeling when working with a priori assumptions on the energy of the system \cite{eidnesPseudoHamiltonianNeuralNetworks2023a,rothStablePortHamiltonianNeural2025}. Identifying the dynamics in port-Hamiltonian form separates the Hamiltonian, dissipation, and input map rather than representing them through a single vector field. This decomposition gives the learned coefficients a direct physical interpretation and makes it possible to modify individual components. This makes it possible, for example, to change the dissipation or applied input when evaluating the identified model.

We introduce two methods. First, SIPHy exploits the assumed port-Hamiltonian structure to cast joint identification of the Hamiltonian, dissipation, and input map as a linear sparse-regression problem, unlike previous approaches \cite{leeStructurepreservingSparseIdentification2022}. Second, Hamiltonian flow splines provide smooth trajectory and derivative estimates from noisy or incomplete observations. We evaluate the combined approach on several synthetic systems and on experimental trajectory data. An illustration of the proposed framework is presented in \cref{fig:overview}, and \cref{tab:hfs_vs_siphy} provides an overview of the two proposed algorithms.

\begin{table}[t]
  \caption{Comparison of Hamiltonian flow splines and SIPHy. $\dot{\mathbf{Q}} := \operatorname{diag}(\dot{\mathbf{q}})$ and $\mathbf{U} := \operatorname{diag}(\mathbf{u})$.}
\label{tab:hfs_vs_siphy}
\centering
\small
\renewcommand{\arraystretch}{1.1}
\resizebox{\textwidth}{!}{
\begin{tabular}{
    >{\raggedright\arraybackslash}p{0.15\linewidth} 
    >{\raggedright\arraybackslash}p{0.38\linewidth}
    >{\raggedright\arraybackslash}p{0.34\linewidth}
}
\toprule
& \textbf{Hamiltonian flow splines} & \textbf{SIPHy} \\
\midrule

\textbf{Purpose}
& Differentiation of noisy data
& System identification \\

\textbf{Structure}
& Piecewise Hamiltonian
& Port-Hamiltonian \\

\textbf{Optimization}
& Constrained nonlinear
& Sequentially thresholded least squares \\

\textbf{Notation}
& Trajectory: $\mathbf{z}(t)$\newline
  Coefficients: $\mathbf{a}^{(k)}$\newline
  Hamiltonian: $H(\mathbf{z};\ \mathbf{a}) = \boldsymbol{\Lambda}(\mathbf{z})^\top \mathbf{a}$
& Trajectory: $\mathbf{y}_w(t)$\newline
  Coefficients: $\mathbf{w}=[\mathbf{c} \; \mathbf{r} \; \mathbf{g}]$ \newline
  Hamiltonian: $H(\mathbf{y}) = \boldsymbol{\Xi}(\mathbf{y})^\top \mathbf{c}$
  \\

\textbf{Coefficients}
& Satisfy continuity and coercivity
& Sparse \\

\textbf{Time domain}
& Local: $\mathbf{a}(t) = \mathbf{a}^{(k)}$ for $t \in [T_k,T_{k+1})$
& Global: $\mathbf{w}(t) = \mathbf{w}$ for $t\in [0,T]$ \\

\textbf{Basis for $H( \cdot  )$}
& $\boldsymbol{\Lambda}(\mathbf{z})=[z_1,\; z_1z_2,\; \dots,\; z_1^4,\; z_2^4]$
& $\boldsymbol{\Xi}(\mathbf{y})=[y_1,\; y_1y_2,\; \dots,\; \sin(y_1)]$ \\

\textbf{Dynamics}
& \makecell[l]{$\displaystyle
\begin{bmatrix}
\dot{\mathbf{q}}\\[2pt]
\dot{\mathbf{p}}
\end{bmatrix}
=
\begin{bmatrix}
\nabla_p \boldsymbol{\Lambda}(\mathbf{z})\\[2pt]
-\nabla_q \boldsymbol{\Lambda}(\mathbf{z})
\end{bmatrix}
\mathbf{a}^{(k)},
$}
& \makecell[l]{$\displaystyle
\begin{bmatrix}
\dot{\mathbf{q}}\\[2pt]
\dot{\mathbf{p}}
\end{bmatrix}
\!=\!
\begin{bmatrix}
\nabla_p \boldsymbol{\Xi}(\mathbf{q},\mathbf{p}) &\! 0 &\! 0\\[2pt]
-\nabla_q \boldsymbol{\Xi}(\mathbf{q},\mathbf{p}) &\! -\dot{\mathbf{Q}} &\! \mathbf{U}
\end{bmatrix}
\!\!\begin{bmatrix}
\mathbf{c}\\[2pt]
\mathbf{r}\\[2pt]
\mathbf{g}
\end{bmatrix}
$} \\

\bottomrule
\end{tabular}}
\vspace{-1em}
\end{table}

\section{Related work}

Recent algorithms for system identification embed geometric priors into regression. In \cite{leeStructurepreservingSparseIdentification2022}, SINDy is unified with neural ODEs to learn Poisson and metriplectic brackets; \cite{chuDiscoveringInterpretableDynamics2020} and \cite{otto2025unified} provides a framework for enforcing symmetry or learning symmetry in dynamic models. The work in \cite{holmsenPseudoHamiltonianSystemIdentification2024} uses gradient-based optimization to identify analytic pseudo-Hamiltonian models whose internal dynamics are separated from damping and external forces, while \cite{dipietroSparseSymplecticallyIntegrated2020} assumes separability of the Hamiltonian and leverages explicit symplectic integrators.

Noise-robust approaches integrate denoising into system identification. In \cite{cortiellaPrioriDenoisingStrategies2022}, the authors compare local and global smoothing algorithms as well as different methods for hyperparameter search.  Denoising and sparse recovery via constrained inverse problems are considered in \cite{hokansonSimultaneousIdentificationDenoising2023},  \cite{sunPhysicsinformedSplineLearning2021} leverages physics-informed spline interpolation, whereas \cite{kahemanAutomaticDifferentiationSimultaneously2022} simultaneously identifies the dynamical system while modeling the noise distribution. An alternative to explicit derivative estimation is to formulate system identification in weak form. Weak SINDy (WSINDy) replaces point wise derivatives by integral relations against test functions, substantially improving robustness to measurement noise \cite{messenger2021weak}. Weak-form sparse identification has also been applied to Hamiltonian systems \cite{messengerCoarseGrainingHamiltonianSystems2023}.

Port-Hamiltonian systems represent dynamics in terms of stored energy, dissipation, and power exchange with the environment \cite{schaftPortHamiltonianSystemsTheory2014}. Their explicit energy-balance structure makes them useful in passivity-based modeling and control \cite{ortegaInterconnectionDampingAssignment2002,schaftPortHamiltonianSystemsTheory2014}. Port-Hamiltonian formulations have been developed for a range of physical systems, including electrical, mechanical, and fluid systems \cite{escobarHamiltonianViewpointModeling1999,rashadPortHamiltonianModelingIdeal2021,forniPortHamiltonianFormulationRigidBody2015a}. More recently, several works learn sparse or adaptive port-Hamiltonian systems directly from measured trajectories while preserving the underlying energetic constraints \cite{rettbergDatadrivenIdentificationLatent2024,junkerAdaptiveDataDrivenModels2025}.

\section{Sparse identification of port-Hamiltonian systems (SIPHy)}

The central observation enabling SIPHy is that, under structural assumptions on the port-Hamiltonian system — namely that dissipation and control only act on the momentum coordinates, and therefore enter only in the equation for $\dot{\mathbf{p}}$ — the joint identification of the Hamiltonian function, the dissipation matrix, and the input map reduces to a single linear regression problem in the unknown coefficients. This is in contrast to prior structure-preserving identification approaches \cite{leeStructurepreservingSparseIdentification2022, holmsenPseudoHamiltonianSystemIdentification2024}, which require nonlinear optimization. The linearity makes sequentially thresholded least squares directly applicable and yields sparse, interpretable models without iterative gradient methods.

\subsection{Port-Hamiltonian systems}

We consider port-Hamiltonian systems of the form 

\begin{align}
\dot{\mathbf{y}}(t) &= \left( \begin{bmatrix}
0 & \mathbf{I}\\
-\mathbf{I} & 0
\end{bmatrix}
- \begin{bmatrix}
0 & 0\\
0 & \mathbf{R}
\end{bmatrix} \right) \nabla H(\mathbf{y}(t)) + \begin{bmatrix}
0 \\
\mathbf{G}
\end{bmatrix}\mathbf{u}(t),
\label{eq:port_hamiltonian}
\end{align}
where $\mathbf{y} \in \mathbb R^n$ is the system state, $\mathbf{u}(t) \in \mathbb R^d$ is a control input, and $n=2d$. For simplicity, assume that $\mathbf{R} = \mathrm{diag} \{r_1,r_2,\dots,r_d\}$ and $\mathbf{G} = \mathrm{diag}\{g_1,g_2,\dots,g_d\}$ are diagonal matrices where $r_i \geq 0$. The system is conservative for $r_i = 0$ and dissipative for $r_i > 0$ with $i = 1,2,\dots,d$. Let $\mathbf{y}(t)$ be the solution and $\tilde{\mathbf{y}}_i = \mathbf{y}(t_i) + \boldsymbol{\delta}_i$ be noisy measurements, where $\boldsymbol{\delta}_i \sim \mathcal N(0,\sigma^2\mathbf{I})$. System identification in this setting aims at identifying the Hamiltonian function $H(\mathbf{y})$ and the matrices $\mathbf{R}$ and $\mathbf{G}$ from the noisy measurements $\tilde{\mathbf{y}}_i$.

\subsection{System identification by least squares}

Let $\boldsymbol{\Xi} : \mathbb R^n \rightarrow \mathbb R^p$ be a vector containing $p$ functions of $\mathbf{y}$, for instance a combination of polynomials and trigonometric terms 
\begin{equation}
    \boldsymbol{\Xi}(\mathbf{y}) = \begin{bmatrix}
     y_1 & y_2 & \dots & y_n & y_1^2 & \dots & \sin(y_1) & \cos(y_2) & \dots
    \end{bmatrix}^{\top} \in \mathbb R^p.
\end{equation}
We aim to find a linear combination  of these functions with weights $\mathbf{c}\in\mathbb R^p$, that can be used to approximate the Hamiltonian as
    $H_{c}(\mathbf{y}) = \boldsymbol{\Xi}(\mathbf{y})^{\top}\mathbf{c}$.
Letting $\nabla \boldsymbol{\Xi}(\mathbf{y}) \in \mathbb R^{n \times p}$ be the Jacobian, the corresponding conservative Hamiltonian system is $\dot{\mathbf{y}} = \mathbf{J} \nabla \boldsymbol{\Xi}(\mathbf{y})\mathbf{c}$.

Let $\mathbf{y} = [\mathbf{q},\mathbf{p}]^{\top}$ be a splitting of the system state, where $\mathbf{q},\mathbf{p} \in \mathbb R^d$ are the generalized position and momentum, assuming that $n = 2d$. Since the dissipation and control terms only act on $\dot{\mathbf{p}}$, we can write the system dynamics as
\begin{align*}
    \begin{bmatrix}
        \dot{\mathbf{q}} \\
        \dot{\mathbf{p}}
        \end{bmatrix}
        &= \left( \begin{bmatrix}
    0 & \mathbf{I}\\
    -\mathbf{I} & 0
    \end{bmatrix}
    - \begin{bmatrix}
    0 & 0\\
    0 & \mathbf{R}
    \end{bmatrix} \right) 
    \begin{bmatrix}
    \nabla_q H(\mathbf{q},\mathbf{p}) \\
    \nabla_p H(\mathbf{q},\mathbf{p})
    \end{bmatrix} 
    + 
    \begin{bmatrix}
    0 \\
    \mathbf{G} 
    \end{bmatrix} \mathbf{u}\\
    &= \begin{bmatrix}
        \nabla_p H(\mathbf{q},\mathbf{p}) \\
        -\nabla_q H(\mathbf{q},\mathbf{p})
    \end{bmatrix}+
    \begin{bmatrix}
        0 \\
        - \mathbf{R}\dot{\mathbf{q}} + \mathbf{G} \mathbf{u}
    \end{bmatrix}.
    \end{align*}

Introducing the coefficients $\mathbf{r},\mathbf{g} \in \mathbb R^d$, we can express system identification of equations of the form \eqref{eq:port_hamiltonian} as solving a linear system of equations. For a given point in time, we aim to find a coefficient vector $\mathbf{w} = \begin{bmatrix}\mathbf{c}^{\top} & \mathbf{r}^{\top} & \mathbf{g}^{\top}\end{bmatrix}^{\top} \in \mathbb R^{p + n}$ such that
\begin{align*}
    \begin{bmatrix}
        \dot{\mathbf{q}} \\
        \dot{\mathbf{p}}
        \end{bmatrix}
        &= \underbrace{
        \begin{bmatrix}
    \nabla_p \boldsymbol{\Xi}(\mathbf{q},\mathbf{p}) & 0 & 0\\
    -\nabla_q \boldsymbol{\Xi}(\mathbf{q},\mathbf{p}) & -\mathrm{diag}(\dot{\mathbf{q}}) &  \mathrm{diag}(\mathbf{u})
\end{bmatrix}}_{\Theta(q,p,\dot q, u)}
\begin{bmatrix}
    \mathbf{c} \\
    \mathbf{r} \\
    \mathbf{g}
\end{bmatrix}\\
&= \boldsymbol{\Theta}(\mathbf{q},\mathbf{p},\dot{\mathbf{q}}, \mathbf{u}) \mathbf{w},
\end{align*}
with $\boldsymbol{\Theta}(\mathbf{q},\mathbf{p},\dot{\mathbf{q}}, \mathbf{u}) \in \mathbb R^{n \times (p + n)}$. In terms of $\mathbf{y}$ we write $\dot{\mathbf{y}} = \boldsymbol{\Theta}(\mathbf{y},\dot{\mathbf{y}},\mathbf{u}) \mathbf{w}$ and 
for $m$ observations in time $\{\mathbf{y}_i,\dot{\mathbf{y}}_i, \mathbf{u}_i\}_{i=1}^m$, by concatenation we obtain
\begin{align*}
    \begin{bmatrix}
        \dot{\mathbf{y}}_1 \\
        \dot{\mathbf{y}}_2 \\
        \vdots \\
        \dot{\mathbf{y}}_m
        \end{bmatrix} &= \begin{bmatrix}
    \boldsymbol{\Theta}(\mathbf{y}_1,\dot{\mathbf{y}}_1,\mathbf{u}_1) \\
    \boldsymbol{\Theta}(\mathbf{y}_2,\dot{\mathbf{y}}_2,\mathbf{u}_2) \\
    \vdots \\
    \boldsymbol{\Theta}(\mathbf{y}_m,\dot{\mathbf{y}}_m,\mathbf{u}_m)
\end{bmatrix} \mathbf{w},
\end{align*}
which leads to the least squares problem
\begin{align}
\min_{\mathbf{w}\in \mathbb R^{p+n}} \| \dot{\mathbf{Y}} - \boldsymbol{\Theta}(\mathbf{Y},\dot{\mathbf{Y}}, \mathbf{U}) \mathbf{w} \|_2^2,
\label{eq:phs_sindy}
\end{align}
where $\mathbf{Y},\dot{\mathbf{Y}} \in \mathbb R^{mn}$ is the concatenated state and its derivative and the evaluations of the basis functions are given by $\boldsymbol{\Theta}(\mathbf{Y},\dot{\mathbf{Y}}, \mathbf{U}) \in \mathbb R^{mn \times (p + n)}$.

\begin{remark}[Extending $\mathbf{R}$ and $\mathbf{G}$ beyond constant diagonal]
The same linear-regression principle extends from constant diagonal \(\mathbf{R}\) to matrix-valued dissipation \(\mathbf{R}(\mathbf{z})\), provided it is parametrized linearly in a known feature library $\{\psi_l\}$,
\[
\mathbf{R}(\mathbf{x})=\sum_{\ell=1}^M \psi_\ell(\mathbf{x})\,\mathbf{R}^{(\ell)},
\qquad \mathbf{x}=(\mathbf{q},\mathbf{p},\dot{\mathbf{q}},\dot{\mathbf{p}}),
\]
where each \(\mathbf{R}^{(\ell)}\in\mathbb R^{d\times d}\) is an unknown constant-coefficient matrix. If we stack the entries of these matrices into a vector
\[
\mathbf{r}
:=
\begin{bmatrix}
\mathrm{vec}(\mathbf{R}^{(1)})^\top\;
\mathrm{vec}(\mathbf{R}^{(2)})^\top\;
\dots \;
\mathrm{vec}(\mathbf{R}^{(M)})^\top
\end{bmatrix}^\top
\in\mathbb R^{Md^2},
\]
then
\[
\mathbf{R}(\mathbf{z})\dot{\mathbf{q}}
=
\sum_{\ell=1}^M \psi_\ell(\mathbf{x})\,\mathbf{R}^{(\ell)}\dot{\mathbf{q}}
=
\Big[
\psi_1(\mathbf{x})(\dot{\mathbf{q}}^\top\otimes \mathbf{I}_d)\;\;
\psi_2(\mathbf{x})(\dot{\mathbf{q}}^\top\otimes \mathbf{I}_d)\;\;
\cdots\;\;
\psi_M(\mathbf{x})(\dot{\mathbf{q}}^\top\otimes \mathbf{I}_d)
\Big]\mathbf{r}.
\]
Hence, the dissipation term is still linear in the unknown coefficients, and the same generalization could be made for the input matrix $\mathbf{G}$.
\end{remark}

\subsection{Sparse regression via Sequential Thresholded Least Squares}

We solve \eqref{eq:phs_sindy} via Sequential Thresholded Least Squares (STLSQ) \cite{brunton_discovering_2016,de2020pysindy,zhang2019convergence}, which alternates ridge regression with hard-thresholding to produce a sparse coefficient vector. STLSQ requires a regularization parameter $\alpha \geq 0$ and a coefficient threshold $\tau > 0$, both selected via the procedure in \cref{sec:hyper_selection}. We choose a Hamiltonian basis $\boldsymbol{\Xi}$ with all unique monomials up to degree $5$ in $\mathbf{y} \in \mathbb R^n$, optionally augmented with the $2n$ trigonometric terms $\{\sin(y_i), \cos(y_i)\}_{i=1}^n$.

The following notation is used to describe the STLSQ algorithm. We define the index set 
    \(
    [P] := \{1,2,\dots,P\},
    \)
    for $P=p+n$, the number of library terms. For any matrix \( \boldsymbol{\Theta} \in \mathbb{R}^{M\times P} \) and any index set \( \mathcal I_S  \subseteq [P] \), we denote by \( \boldsymbol{\Theta}_S \in \mathbb{R}^{M\times S}\) the submatrix formed by the columns of \( \boldsymbol{\Theta} \) indexed by \( \mathcal I_S \). Similarly, for a vector \( \mathbf{w} \in \mathbb{R}^{P} \), \( \mathbf{w}_S \in \mathbb R^S \) denotes the subvector restricted to the indices in \( \mathcal I_S \). STLSQ as implemented in \cite{de2020pysindy} is given by \cref{alg:stlsq} with the following additional step to guarantee nonnegative dissipation.

Neither the least-squares problem \eqref{eq:phs_sindy} nor STLSQ constrains the sign of the dissipation coefficients, so the identified $\mathbf{R}$ may fail to be positive semi-definite under noise, in which case the identified model is not passive. Since $\mathbf{R}=\mathrm{diag}(\mathbf{r})$, we need to enforce $d$ bound constraints $\mathbf{r}\ge 0$. Writing $\mathcal I_R := \{p+1,\dots,p+d\}$ for the indices of $\mathbf{r}$ in $\mathbf{w}$, we keep the coefficients returned by \cref{alg:stlsq} whenever they satisfy $\mathbf{r}\ge 0$. If instead $r_i<0$ for some $i$, we recompute the coefficients on the converged support $\mathcal I_S$ by solving the bound-constrained problem
\begin{equation}
\mathbf{w}_S^{\star} \;=\;
\arg\min_{\mathbf{w}_S \,:\; w_i \ge 0 \;\forall i \in \mathcal I_S \cap \mathcal I_R}
\;\; \bigl\| \dot{\mathbf{Y}} - \boldsymbol{\Theta}_S\,\mathbf{w}_S \bigr\|_2^2
\;+\; \alpha\,\|\mathbf{w}_S\|_2^2 ,
\label{eq:refit}
\end{equation}
a convex quadratic program in $|\mathcal I_S|$ variables with at most $d$ active bounds. This guarantees $\mathbf{R}\succeq 0$ for the returned model. The refit is applied only to the converged support and only when the unconstrained solution is negative, so it leaves both the support selection of \cref{alg:stlsq} and any nonnegative solution unchanged.

\begin{algorithm}[!htb]
    \caption{Sequential Thresholded Least Squares (STLSQ) with nonnegative refit}
    \label{alg:stlsq}
    \begin{algorithmic}[1]
    \Require Feature matrix \(\boldsymbol{\Theta} \in \mathbb{R}^{M\times P}\), derivative data \(\dot{\mathbf{Y}} \in \mathbb{R}^{M}\), regularization parameter \(\alpha\ge0\), threshold \(\tau>0\), maximum iterations \(K_{\max}\).
    \State \textbf{Initialization:} Compute
    $
    \mathbf{w}^{(0)} \gets \arg\min_{\mathbf{w}} \left\{ \|\dot{\mathbf{Y}} - \boldsymbol{\Theta}\,\mathbf{w}\|_2^2 + \alpha\,\|\mathbf{w}\|_2^2 \right\}.
    $
    \For{$k=0,1,\dots,K_{\max}$}
        \State
        $\mathcal I_S \gets \{ i \; \mathrm{s.t.}\; |w^{(k)}_i| \ge \tau \}.$
        \Comment{Update support set.}
        \State
        \(
        \mathbf{w}^{(k+1)} \gets \arg\min_{\mathbf{w}_S} \|\dot{\mathbf{Y}} - \boldsymbol{\Theta}_S\,\mathbf{w}_S\|_2^2 + \alpha\,\|\mathbf{w}_S\|_2^2,
        \) \Comment{Ridge regression.}
        \State $\mathcal I_S^{(k+1)} \gets \mathcal I_S$.
        \If{$\mathcal I_S^{(k+1)} = \mathcal I_S^{(k)}$}
            \State \textbf{break}
        \EndIf
    \EndFor
    \If{$w^{(k+1)}_i < 0$ for some $i \in \mathcal I_S^{(k+1)} \cap \mathcal I_R$}
        \State
        \(
        \mathbf{w}^{(k+1)} \gets \arg\min_{\mathbf{w}_{S}\,:\; w_i \ge 0\; \forall i \in \mathcal I_S \cap \mathcal I_R} \|\dot{\mathbf{Y}} - \boldsymbol{\Theta}_{S}\,\mathbf{w}_{S}\|_2^2 + \alpha\,\|\mathbf{w}_{S}\|_2^2,
        \) \Comment{Refit.}
    \EndIf
    \State \Return \(\mathbf{w}^{(k+1)}\)
    \end{algorithmic}
    \end{algorithm}

\section{Hamiltonian flow splines}\label{sec:energy_splines}

We now introduce Hamiltonian flow splines: a structure-preserving smoother for trajectory data from port-Hamiltonian systems. While used as a preprocessing step for SIPHy in this paper, flow splines are an independent contribution that can be paired with any downstream identification method requiring smooth derivatives. A Hamiltonian flow spline is a trajectory $\mathbf{z}(t) : [0,T] \rightarrow \mathbb R^n$ obtained from solving a piecewise polynomial Hamiltonian system. Here, we derive linear conditions on the polynomial coefficients ensuring $\mathcal C^3$ continuity and bounded trajectories through coercive Hamiltonians. We first define the piecewise polynomial Hamiltonians, and then linear conditions for continuity and coercivity. The optimization problem to fit the flow splines to data is described in \cref{sec:optimization}.

\subsection{Piecewise polynomial Hamiltonians}

The Hamiltonians used for the flow splines are constructed by taking linear combinations of monomials of $\mathbf{z} \in \mathbb R^{n}$ with coefficients $a_i \in \mathbb R$:
\[
H(\mathbf{z}) = a_1 z_1 + a_2 z_1 z_2 + a_3 z_1^2 + \cdots
\]
To enforce bounded trajectories (described in detail in \cref{sec:coercivity}), we restrict the set of highest-order monomials to be even and contain only non-mixed powers $\{z_1^p,z_2^p,\dots,z_{n}^p\}$ with order $p=2l$ for any integer $l\geq 1$. Define the multi-index sets 
\[
\mathcal A_{p-1}:=\{\boldsymbol{\alpha}\in\mathbb N^n:\ 0<|\boldsymbol{\alpha}|\le p-1\},
\qquad
\mathcal E_p:=\{p \mathbf{e}_1,\dots,p \mathbf{e}_{n}\},
\]
where \(\mathbf{e}_i\in\mathbb R^n\) denotes the \(i\)-th canonical basis vector and $|\boldsymbol{\alpha}| = \sum_i \alpha_i$. For example $\boldsymbol{\alpha} = (i,j)$ gives $\mathbf{z}^{\boldsymbol{\alpha}} = z_1^iz_2^j$. We then use the restricted index set
\begin{equation}
    \mathcal R_p:=\mathcal A_{p-1}\cup \mathcal E_p,
    \label{eq:index_set}
\end{equation}
and define the monomial basis as
\[
\boldsymbol{\Lambda}(\mathbf{z}):=\big[\mathbf{z}^{\boldsymbol{\alpha}}\big]_{\alpha \in \mathcal R_p} : \mathbb R^n \rightarrow \mathbb R^{N_\Lambda},
\qquad
N_\Lambda=\binom{n+p-1}{p-1} + n -1.
\]
For a coefficient vector \(\mathbf{a}\in\mathbb{R}^{N_\Lambda}\), we define the parametrized Hamiltonian
\[
H(\mathbf{z};\,\mathbf{a}):=\boldsymbol{\Lambda}(\mathbf{z})^\top \mathbf{a} .
\]
Consider now a piecewise-constant coefficient
function \(\mathbf{a}(t)\) over $K$ segments in time of the form
\[
\mathbf{a}(t)=\mathbf{a}^{(k)},\qquad \text{for } t\in[T_k,T_{k+1}),\quad k=0,\dots,K-1,
\]
with \(\mathbf{a}^{(k)}\in\mathbb{R}^{N_\Lambda}\) and where  $0<T_1< \dots < T_K = T$ are \textit{knots} defining the length of the segments.
The Hamiltonian flow spline \(\mathbf{z}(t)\) is then defined as the solution of the time-varying Hamiltonian system
\begin{equation}
\dot{\mathbf{z}}(t)=\mathbf{J}\,\nabla_z H\!\left(\mathbf{z}(t);\,\mathbf{a}(t)\right),\qquad \mathbf{z}(0)=\mathbf{z}_0.
\label{eq:hamiltonian_sys_at}
\end{equation}
 Global \(\mathcal C^3\) regularity across knots is enforced by the continuity conditions in the following section, while coercivity guaranteeing bounded trajectories is discussed in \cref{sec:coercivity}.

\subsection{Continuity}\label{subsec:continuity}

Because \(H(\mathbf{z};\mathbf{a})\) is linear in \(\mathbf{a}\), the derivatives of the basis that determine the regularity of \(\mathbf{z}(t)\) across a knot are linear in \(\mathbf{a}\) as well. Fix a knot \(T_k\), write \(\mathbf{z}_k:=\mathbf{z}(T_k)\), and denote by
\[\dot{\mathbf{z}}_k^-  := \lim_{h \rightarrow 0} \frac{\mathbf{z}(T_k)- \mathbf{z}(T_k - h)}{h}, \quad   \dot{\mathbf{z}}_k^+ := \lim_{h \rightarrow 0} \frac{\mathbf{z}(T_k + h) - \mathbf{z}(T_k)}{h}\]
the one-sided first derivatives at \(T_k\) and similarly for the second derivatives, replacing $\mathbf{z}$ by $\dot{\mathbf{z}}$ on the right-hand sides. Define the \(n\times N_\Lambda\) matrices
\begin{align*}
\boldsymbol{\Lambda}_k^1 &:= \nabla_z \boldsymbol{\Lambda}(\mathbf{z}_k), \\
\boldsymbol{\Lambda}_k^2 &:= \nabla_z^2 \boldsymbol{\Lambda}(\mathbf{z}_k)\!\left(\dot{\mathbf{z}}_k^-,\;\cdot\;\right), \\
\boldsymbol{\Lambda}_k^3 &:= \nabla_z^3 \boldsymbol{\Lambda}(\mathbf{z}_k)\!\left(\dot{\mathbf{z}}_k^-,\dot{\mathbf{z}}_k^-,\;\cdot\;\right)
\;+\; \nabla_z^2 \boldsymbol{\Lambda}(\mathbf{z}_k)\!\left(\ddot{\mathbf{z}}_k^-,\;\cdot\;\right),
\end{align*}
where $\nabla_z^k \boldsymbol{\Lambda} : \mathbb R^n \times \cdots \times  \mathbb R^n \rightarrow \mathbb R^{N_\Lambda}$ are multilinear forms, contracted to matrices in the equations above.
The continuity-constraint matrix is defined as 
\begin{equation}
\mathbf{C}_k := \begin{bmatrix}
\boldsymbol{\Lambda}_k^1 \\[2pt] \boldsymbol{\Lambda}_k^2 \\[2pt] \boldsymbol{\Lambda}_k^3
\end{bmatrix} \in \mathbb{R}^{3n\times N_\Lambda}.
\label{eq:C_k}
\end{equation}

\begin{theorem}[Linear continuity condition]\label{thm:C3}
Let \(\mathbf{a}(t)\) be piecewise constant such that \(\mathbf{a}(t)=\mathbf{a}^{(k)}\) on \([T_k,T_{k+1})\), and let
\(\mathbf{z}:[0,T]\to\mathbb R^n\) be continuous and satisfying
\eqref{eq:hamiltonian_sys_at} on each subinterval \((T_k,T_{k+1})\).
Let \(\mathbf{C}_k\) be given by \eqref{eq:C_k}. If, for every knot \(T_k\),
\begin{equation}
\mathbf{C}_k\big(\mathbf{a}^{(k+1)}-\mathbf{a}^{(k)}\big)=0,
\label{eq:constraints_C}
\end{equation}
then \(\mathbf{z}\in \mathcal C^3([0,T])\).
\end{theorem}

\begin{proof}
On each open interval \((T_{k},T_{k+1})\), $\mathbf{a}(t)$ is constant, so \(\nabla_z H(\mathbf{z}\; ;\mathbf{a}^{(k)})\) is \(\mathcal C^2\) and the solution $\mathbf{z}(t)$ is thus \(\mathcal C^3\) on \((T_{k},T_{k+1})\).
Define
\(
\mathbf{a}^-:=\mathbf{a}^{(k)}\) and \(\mathbf{a}^+:=\mathbf{a}^{(k+1)}
\).
The one-sided first derivatives are
\[
\dot{\mathbf{z}}_k^\pm = \mathbf{J}\nabla_z\boldsymbol{\Lambda}(\mathbf{z}_k)\,\mathbf{a}^\pm.
\]
From \eqref{eq:constraints_C} we have that
 \[
\nabla_z\boldsymbol{\Lambda}(\mathbf{z}_k)(\mathbf{a}^+-\mathbf{a}^-)=0,
\]
and thus \(\dot{\mathbf{z}}_k^-=\dot{\mathbf{z}}_k^+\), which we denote by \(\dot{\mathbf{z}}_k\). Similarly, since $\ddot{\mathbf{z}}_k^{\pm} = \mathbf{J}\nabla_z^2\boldsymbol{\Lambda}(\mathbf{z}_k)(\dot{\mathbf{z}}_k,\mathbf{a}^{\pm})$ then \[
\nabla_z^2\boldsymbol{\Lambda}(\mathbf{z}_k)\!\left(\dot{\mathbf{z}}_k,\;\mathbf{a}^+-\mathbf{a}^-\right)=0,
\]
which implies \(\ddot{\mathbf{z}}_k^-=\ddot{\mathbf{z}}_k^+\) denoted by \(\ddot{\mathbf{z}}_k\). Finally,
\[
\dddot{\mathbf{z}}_k^\pm =
\mathbf{J}\nabla_z^3\boldsymbol{\Lambda}(\mathbf{z}_k)\!\left(\dot{\mathbf{z}}_k,\dot{\mathbf{z}}_k,\mathbf{a}^\pm\right)
+\mathbf{J}\nabla_z^2\boldsymbol{\Lambda}(\mathbf{z}_k)\!\left(\ddot{\mathbf{z}}_k,\mathbf{a}^\pm\right),
\]
and the last $n$ elements of \eqref{eq:constraints_C} means that \(\dddot{\mathbf{z}}_k^-=\dddot{\mathbf{z}}_k^+\).
\end{proof}

\begin{remark}[Parameterization of admissible coefficient updates]
If \(\mathbf{C}_k\) has full row rank \(3n\) and \(N_\Lambda>3n\), then the admissible differences \(\mathbf{a}^{(k+1)}-\mathbf{a}^{(k)}\) lie in \(\mathrm{null}(\mathbf{C}_k)\).
Let \(\mathbf{U}\boldsymbol{\Sigma}\mathbf{V}^\top=\mathbf{C}_k\) be a singular value decomposition and let \(\tilde{\mathbf{V}}_k\in\mathbb{R}^{N_\Lambda\times (N_\Lambda-3n)}\) collect the last \(N_\Lambda-3n\) columns of \(\mathbf{V}\).
Then any coefficient update preserving \(\mathcal C^3\) can be written as
\[
\mathbf{a}^{(k+1)} \;=\; \mathbf{a}^{(k)} \;+\; \tilde{\mathbf{V}}_k\,\tilde{\mathbf{a}}^{(k)},\qquad \tilde{\mathbf{a}}^{(k)} \in\mathbb{R}^{N_\Lambda-3n}.
\]
\end{remark}

\subsection{Coercivity of the Hamiltonian}\label{sec:coercivity}

A practical pitfall of polynomial-based smoothers is unbounded extrapolation: small fitting errors can lead to trajectories that diverge, particularly under noise. We address this directly through coercivity, i.e.\ the property that
\[
\lim_{\|\mathbf{z}\|\to\infty} H(\mathbf{z})=+\infty.
\]

By \cite[Prop.~11.11]{bauschke2017}, coercivity is equivalent to the boundedness of all sublevel sets. Along a Hamiltonian trajectory, the value of
\(H(\mathbf{z})\) is conserved, so the trajectory remains in a level set of \(H\). Thus, a coercive Hamiltonian implies bounded trajectories. For the restricted polynomial family used here, we show that coercivity reduces to a linear inequality on the leading coefficients, making it simple to enforce during optimization:

\begin{theorem}[Linear coercivity condition]\label{thm:coercivity}
Let \(p=2l\) and let \(H(\mathbf{z};\mathbf{a})=\boldsymbol{\Lambda}(\mathbf{z})^\top \mathbf{a}\) be built from the
restricted monomial basis given by $\mathcal R_p$ in \eqref{eq:index_set}. Assume that $L<N_{\Lambda}$ is the index such that  $a_{L+j}$ are the weights for the degree $p$ monomials, i.e. 
\[
H(\mathbf{z};\mathbf{a})= \sum_{i=1}^L a_i\mathbf{z}^{\boldsymbol{\alpha}_i} + \sum_{j=1}^{n} a_{L+j} z_j^p.
\]
If 
\begin{equation}
a_{L+j}\ge \epsilon >0, \quad j=1,\dots,n,
\label{eq:coercivity_constraint}
\end{equation}
then \(H(\mathbf{z};\mathbf{a})\) is coercive, i.e.
\[
\lim_{\|\mathbf{z}\|\to\infty} H(\mathbf{z};\mathbf{a})=+\infty.
\]
\end{theorem}

\begin{proof}
Since \(p\) is even,
\[
\sum_{j=1}^n a_{L+j}\, z_j^p
\ge
\epsilon \sum_{j=1}^n |z_j|^p
=
\epsilon \|\mathbf{z}\|_{p}^{p} .
\]
For the remaining polynomial
\(\sum_{i=1}^L a_i\mathbf{z}^{\boldsymbol{\alpha}_i}\),
we have $|\boldsymbol{\alpha}_i| \leq p-1$, and thus
\[
|\mathbf{z}^{\boldsymbol{\alpha}}|
=
\prod_{j=1}^n |z_j|^{[\boldsymbol{\alpha}]_j}
\le
C(1+\|\mathbf{z}\|^{p-1}),
\]
for a constant $C>0$. Therefore
\begin{align*}
H(\mathbf{z};\mathbf{a})
&\ge
\epsilon \|\mathbf{z}\|_p^{p} - C\bigl(1+\|\mathbf{z}\|^{p-1}\bigr)\\
&\ge \epsilon c_p \|\mathbf{z}\|^p - C\bigl(1+\|\mathbf{z}\|^{p-1}\bigr),
\end{align*}
with \(c_p>0\) being the constant such that \(\|\mathbf{z}\|_p^p \ge c_p \|\mathbf{z}\|^p\) by the equivalence of norms on \(\mathbb R^n\).
Hence, 
\(H(\mathbf{z};\mathbf{a})\to +\infty\)
as \(\|\mathbf{z}\|\to\infty\).
\end{proof}

\subsection{Applying Hamiltonian flow splines to data}
\label{sec:optimization}
Consider noisy measurements $\tilde{\mathbf{y}}_n = \mathbf{y}(t_n) + \boldsymbol{\delta}_n$, for $n = 1,2,\dots,N$, and let $\eta \in [0,1)$ be the rate of missing data, meaning $\eta  N$ points in time are missing. We fit a smooth trajectory $\mathbf{z}(t)$ to noisy data, from which the derivative can be evaluated. That is, we want to find coefficients $\mathbf{a}(t)$ that minimize
\begin{equation}
    \begin{aligned}
    \min_{\mathbf{a}(t)}  \sum_{n=1}^N \| \mathbf{z}(t_n) - \tilde{\mathbf{y}}_n \|_2^2 \quad
    \text{subject to} \quad \mathbf{z}(t) \text{ solving }  \eqref{eq:hamiltonian_sys_at}.
\label{eq:optimal_control}
\end{aligned}
\end{equation}

We solve the optimization problem \eqref{eq:optimal_control} by dividing the time interval $[0,T]$ into $K$ intervals $[T_k,T_{k+1})$ with $M$ points in time within each and solve for $\mathbf{a}^{(k)}$ on each interval. This is a greedy algorithm, meaning we approximate the local optimal solution for each interval, which may not be the global optimal solution. A heuristic that mitigates overfitting an individual interval is to  optimize $\mathbf{a}^{(k)}$ over an extended window $[T_{k},T_{k+r})$ with  $r\in\{1,\dots,K-k\}$, but only use the solution $\mathbf{a}^{(k)}$ in the time-interval $[T_{k},T_{k+1})$. The final window may contain fewer points because they are truncated at $T$. The hyperparameters of this algorithm are $M$, $r$, as well as $\lambda$, which is the weight for the $l_2$ regularization on $\mathbf{a}^{(k)}$.

For the first interval $[0,T_{r})$, we optimize over the initial value $\mathbf{z}_0$ \cite{chenSymplecticRecurrentNeural2019} as well as $\mathbf{a}^{(0)}$:
\begin{equation}
    \begin{aligned}
    \min_{\mathbf{z}_0,\mathbf{a}^{(0)}}  \sum_{t_n \in [0,T_{r})} &\| \mathbf{z}_n - \tilde{\mathbf{y}}_n \|_2^2 + \lambda \| \mathbf{a}^{(0)} \|_2^2
   \quad & & \text{subject to}\\[1em]
     \mathbf{z}_n  \text{ solving }  \eqref{eq:hamiltonian_sys_at} \text{ with } H&=H(\mathbf{z};\mathbf{a}^{(0)}) \text{ and } \mathbf{z}(0) = \mathbf{z}_0 & & \text{(ODE)}\\ 
    a_{L+j}^{(0)}&\ge \epsilon, \quad j=1,\dots,n & & \text{(Coercivity)}
\end{aligned}
\label{eq:optimal_iso}
\end{equation}
For the remaining intervals $k\geq1$, we take $\mathbf{z}(T_{k})$, the end point from the previous time segment, as initial value, and ensure $\mathcal C^3$ continuity by requiring $\mathbf{a}^{(k)}$ to satisfy the constraints in \eqref{eq:constraints_C}:
\begin{equation}
    \begin{aligned}
    \min_{\mathbf{a}^{(k)}}  \sum_{t_n \in [T_{k},T_{k+r})}& \| \mathbf{z}_n - \tilde{\mathbf{y}}_n \|_2^2 + \lambda \| \mathbf{a}^{(k)} \|_2^2
   \quad & & \text{subject to}\\[1em]
     \mathbf{z}_n  \text{ solving }  \eqref{eq:hamiltonian_sys_at} \text{ with }  H&=H(\mathbf{z};\mathbf{a}^{(k)}), & & \text{(ODE)}\\ 
    a_{L+j}^{(k)}&\ge \epsilon, \quad j=1,\dots,n, & & \text{(Coercivity)} \\
     \mathbf{C}_k (\mathbf{a}^{(k)} - \mathbf{a}^{(k-1)}) &= 0.  & & \text{(Continuity)}
\end{aligned}
\label{eq:optimal_control_k}
\end{equation}
If observations are missing ($\eta >0$), the time indices $t_j$ where $\tilde{\mathbf{y}}_j$ is missing are excluded from the sum in the optimization objectives in \eqref{eq:optimal_iso} and \eqref{eq:optimal_control_k}.

\section{Numerical experiments}

\begin{table}[!t]
  \caption{System-wise comparison of identification errors. Bold indicates the smaller numeric error in each row (SIPHy vs SINDy). Conservative trajectories to the left and trajectories with dissipation to the right.}
\label{tab:noise_diss}
\centering
\resizebox{\textwidth}{!}{
\input{tables/system_table_compare_0.0_0.1_final}}
\end{table}

We perform numerical experiments on port-Hamiltonian systems with and without dissipation, with varying levels of noise $\gamma \in [0,0.2]$ and rates of missing data $\eta \in [0,0.8]$. We define the perturbations such that $\gamma$ is a noise-to-signal ratio, as described in \cref{sec:noise_level}. The dissipation matrix is chosen to be $\mathbf{R} = c\mathbf{I}$ with $c \in \{0,0.1\}$, where $\mathbf{I}$ is the identity matrix. Trajectories are obtained by integrating the port-Hamiltonian systems with step size $h = 0.01$ for $t_i \in [0,20]$. All experiments identify port-Hamiltonian systems on a single (training) trajectory, as well as testing on one trajectory arising from a new initial value.

The following Hamiltonians are used in the experiments:
\begin{align*}
H(\mathbf{y}) &= \frac{1}{2}y_2^2 - \cos(y_1),
& \text{(Pendulum)}\\
H(\mathbf{y}) &= \frac{1}{4}(y_1^4 + y_2^4),
& \text{(Cubic)}\\
H(\mathbf{y}) &= \frac{1}{2}(y_1^2 + y_2^2 + y_3^2 + y_4^2)
+ y_1^2 y_2 - \frac{1}{3}y_2^3,
& \text{(Hénon--Heiles)}\\
H(\mathbf{y}) &= \frac{1}{4}(y_1 - y_2)^4 + \frac{1}{4}(y_1 + y_2)^4
+ \frac{1}{2}(y_3^2 + y_4^2 + y_2^2),
& \text{(FPUT)}
\end{align*}
in addition to a mass--spring chain described below. 

The experiments use four protocols: varying noise and missing data (\cref{fig:noise_miss}), varying noise and damping (\cref{tab:noise_diss}), while the mass--spring chain has a fixed damping of $\mathbf{R}=0.15\cdot \mathbf{I}_d$ and varying noise level with results in \cref{fig:rollouts} and \cref{tab:mass_spring}. Finally, we use the proposed framework on experimental measurements from \cite{schmidt2009distilling}.
\begin{figure}[!t]
\centering
\includegraphics[width=1\textwidth]{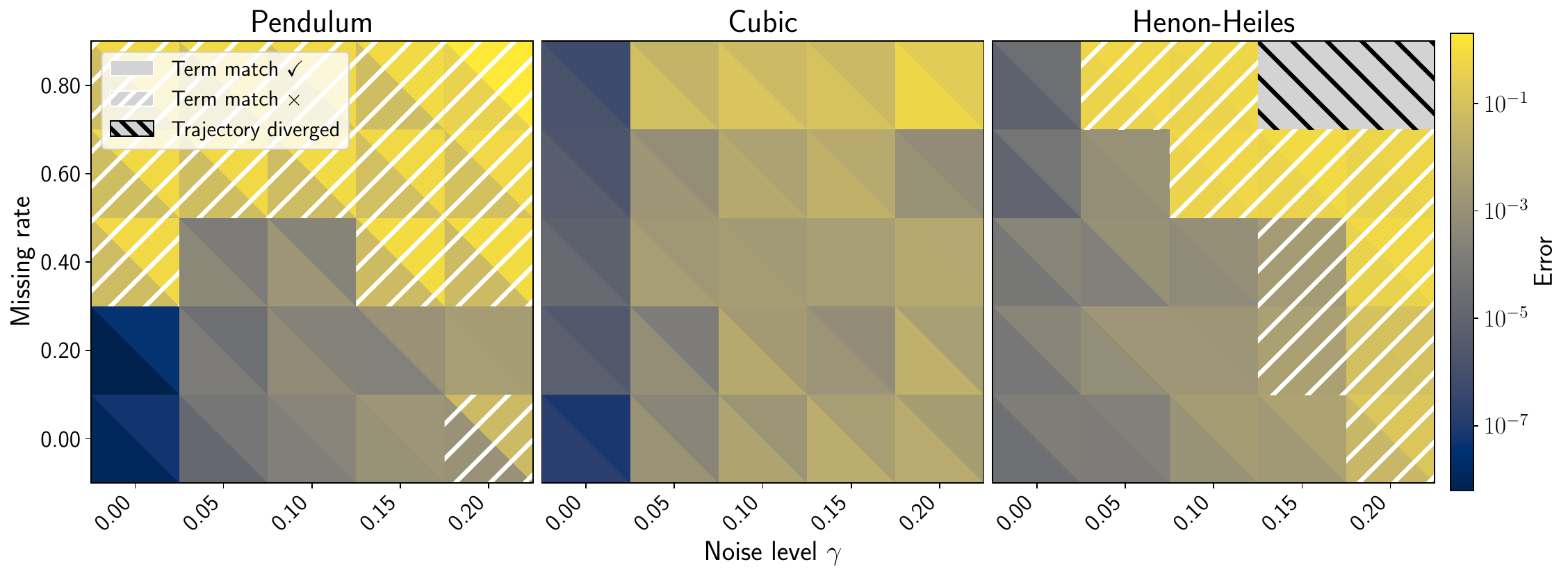}
\caption{SIPHy results showing training/testing NMSE in the lower left / upper right triangular color panels. White lines indicate that the incorrect basis was identified, and black lines indicate that the identified system diverged when integrated.}
\label{fig:noise_miss}
\end{figure}

For the experiments with zero noise presented in \cref{tab:noise_diss}, we used SIPHy and SINDy directly on the trajectory and its numerical derivative. The trajectories generated from this experiment are found in  \cref{fig:rollouts} and \cref{fig:full_plot} in \cref{sec:add_results}.

\Cref{fig:noise_miss} shows training and testing NMSE for three benchmark systems across a grid of noise levels and missing-data rates. Each cell reports two values: the lower-left triangle shows training error and the upper-right triangle shows test error. In two cases with high rates of missing data and noise, the identified system diverges when integrated numerically, as seen in \cref{fig:noise_miss}. The equations describing the identified systems could be found in \cref{tab:coefficient_table} and \cref{tab:mass_spring_identified} in \cref{sec:add_results}.

We measure the error of the approximation $\mathbf{x}_i \approx \mathbf{y}_i$ as the normalized mean squared error (NMSE)
\begin{equation*}
    e(\mathbf{x},\mathbf{y}) := \frac{1}{N \sigma_y^2 }\sum_{i=1}^N  \| \mathbf{x}_i - \mathbf{y}_i \|_2^2,
\end{equation*}
where $\sigma_y^2$ is the total variance defined in \eqref{eq:total_variance} in \cref{sec:noise_level}. The training NMSE measures the discrepancy of the numerical flow of the true and learned system from the initial value generating the training data. Test NMSE considers an unseen trajectory. The initial values used are found in \cref{tab:inits} in \cref{sec:details_num}, together with additional details of the experiments. If the terms in the identified port-Hamiltonian system match the terms of the true system, we denote this by \textit{term match}, which simply means that the basis weights $\mathbf{w}_{\mathrm{id}}$ from STLSQ satisfy $\mathrm{supp}(\mathbf{w}_{\mathrm{true}}) = \mathrm{supp}(\mathbf{w}_{\mathrm{id}})$. The coefficient error is defined as 
$\|\mathbf{w}_{\mathrm{id}}-\mathbf{w}_{\mathrm{true}}\|_2$.

\subsection{Mass--spring chain with boundary actuation}\label{subsec:msd_chain_2n}

We consider a $d$-mass--spring chain where $\mathbf{q} \in \mathbb R^{d}$ denotes the displacements of the masses and $\mathbf{p} \in \mathbb R^{d}$ the corresponding momenta. The Hamiltonian is given by
\[
H(\mathbf{q},\mathbf{p})=\tfrac12 \mathbf{p}^\top \mathbf{M}^{-1}\mathbf{p}+\tfrac12 \mathbf{q}^\top \mathbf{K}\mathbf{q},
\] 
with diagonal mass matrix $\mathbf{M} \succ 0$ and a tri-diagonal stiffness matrix $\mathbf{K}=\mathbf{K}^\top\succeq 0$ encoding neighbour coupling, where we assume that the first mass to the left is connected to a wall by a spring with stiffness $k_1$. This results in a tri-diagonal stiffness matrix with diagonal $K_{i,i} = k_i + k_{i+1}$, off-diagonal  $K_{i,i+1} = K_{i+1,i} = -k_{i+1}$ for $i = 1,\dots,d-1$, and $K_{dd} = k_d$.
\begin{wrapfigure}[7]{hr}{0.4\textwidth}
\input{mass_spring_damper_fig}
\caption{Mass-spring chain with $d=3$ masses.}
\label{fig:mass_spring_damper}
\end{wrapfigure}
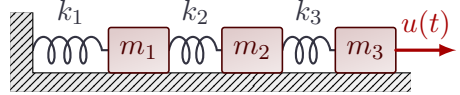
\begin{figure}[!t]
\centering
\includegraphics[width=0.45\textwidth]{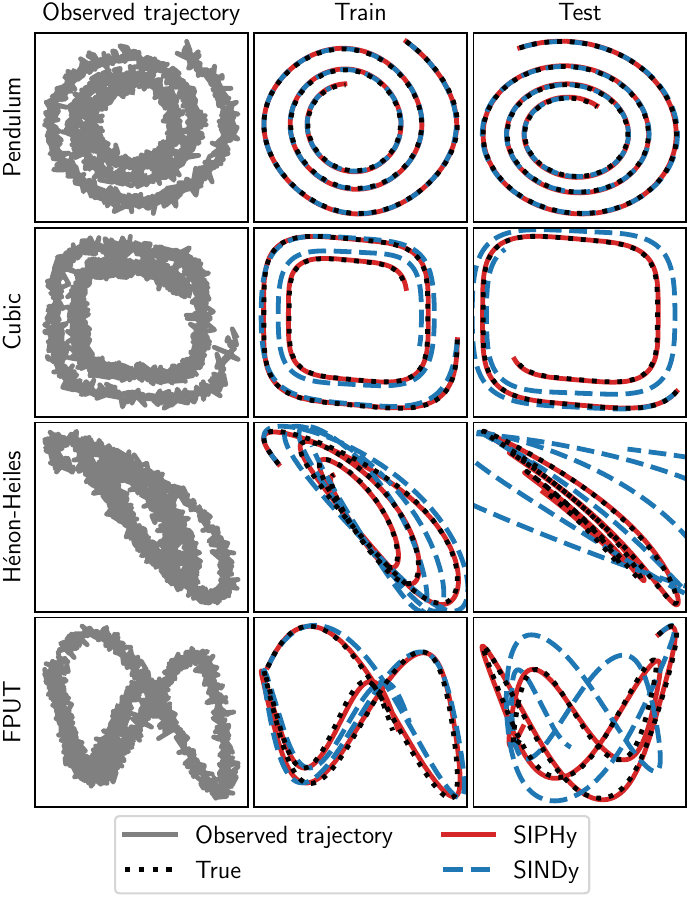}
\includegraphics[width=0.5\textwidth]{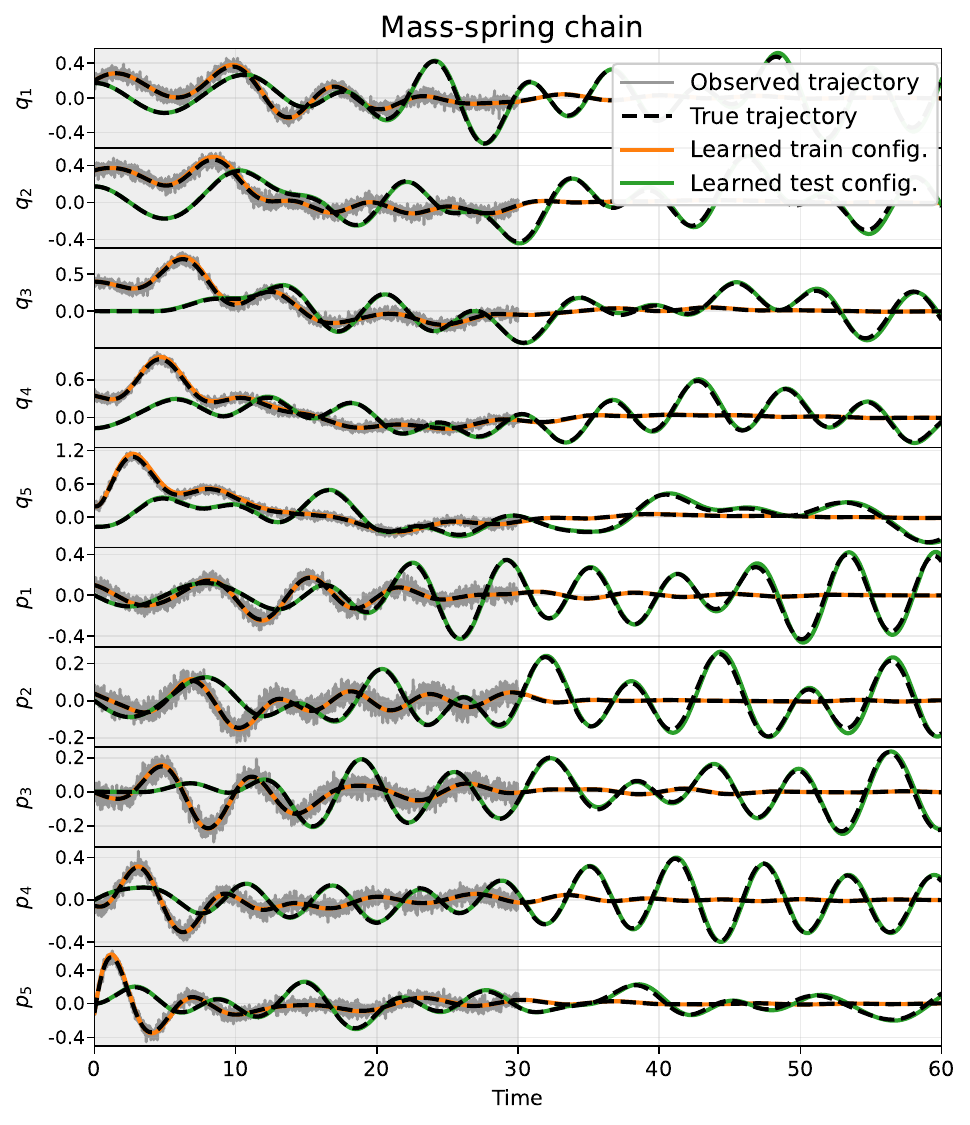}
\caption{\textbf{Left}: Observed and learned trajectories from training and test initial values $\mathbf{y}_0$ for the proposed SIPHy algorithm and SINDy. All the true port-Hamiltonian systems are dissipative with noise level $\gamma = 0.1$. \textbf{Right}: Learned trajectories from training and test configurations. The noisy trajectory ($\gamma = 0.2$) is observed in a time interval $t\in [0,30]$ as indicated by the gray background color.}
\label{fig:rollouts}
\end{figure}

The system is illustrated in \cref{fig:mass_spring_damper} and its dynamics are given by a port-Hamiltonian system of the form \eqref{eq:port_hamiltonian} with 
\(
\mathbf{R} = \mathrm{diag}(r_1,\ldots,r_d)\succeq 0,
\)
corresponding to ground friction. Setting $\mathbf{G}=\mathbf{e}_{d} = [0,\dots,0,1]^{\top}\in\mathbb{R}^{d}$, we allow for actuation of the final mass with a scalar control input $u(t)$. For the numerical experiments, we consider $d=5$ masses with $k_i = 0.4$ for $i=1,\dots,d$, $\mathbf{M}^{-1} = \mathbf{I}_d$.

We apply flow splines and SIPHy on a single trajectory in a time interval up to $T = 30$ with step size $h= 0.01$ and noise levels $\gamma \in \{0.0,0.1,0.2 \}$. The identified system is simulated with the known control (orange trajectory) before we test the identified Hamiltonian with zero dissipation and a sine control term  (green trajectory) in \cref{fig:rollouts}. Both configurations are described in \cref{tab:train-test-settings} in \cref{sec:details_num}.

\begin{table}[h]
  \caption{Mass--spring identification errors across noise levels
$\gamma$ using SIPHy.}
\label{tab:mass_spring}
\centering
\renewcommand{\arraystretch}{1.1}
\setlength{\tabcolsep}{3pt}

\begin{tabular}{lcccc}
\toprule
$\gamma$ & Train Error & Test Error & Coefficient Error & Term Match \\
\midrule
0.0 & 1.2e-04 & 1.3e-04 & 8.0e-03 & $\checkmark$ \\
0.1 & 3.3e-04 & 4.8e-04 & 2.6e-02 & $\checkmark$ \\
0.2 & 8.0e-03 & 1.1e-02 & 1.7e-01 & $\times$ \\
\bottomrule
\end{tabular}
\end{table}

\begin{figure}[!htb]
\centering

\begin{minipage}[t]{0.32\textwidth}
\vspace{0pt}
\centering
\includegraphics[width=\linewidth]
    {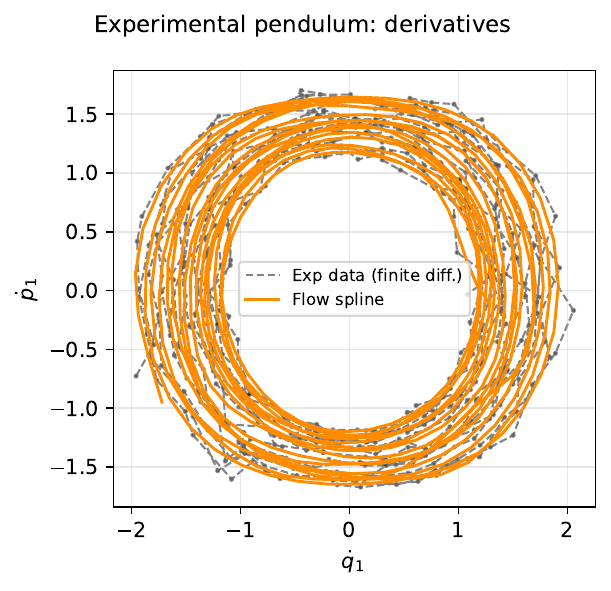}

\vspace{0.8em}

\includegraphics[width=\linewidth]
    {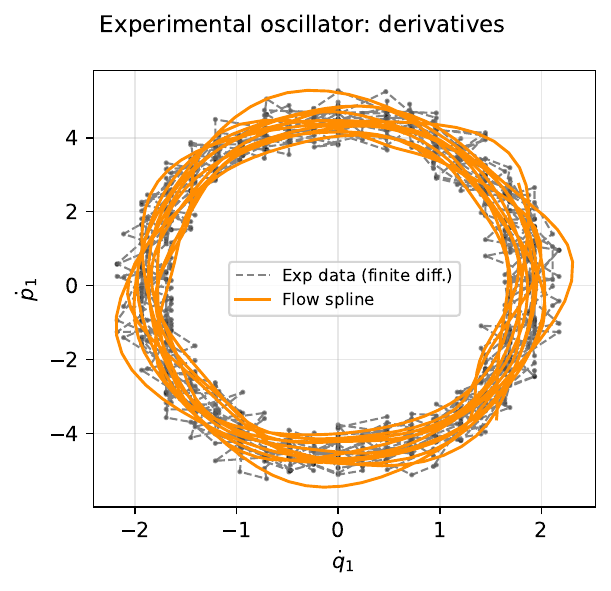}
\end{minipage}
\hfill
\begin{minipage}[t]{0.67\textwidth}
\vspace{0pt}
\centering
\includegraphics[width=\linewidth]
    {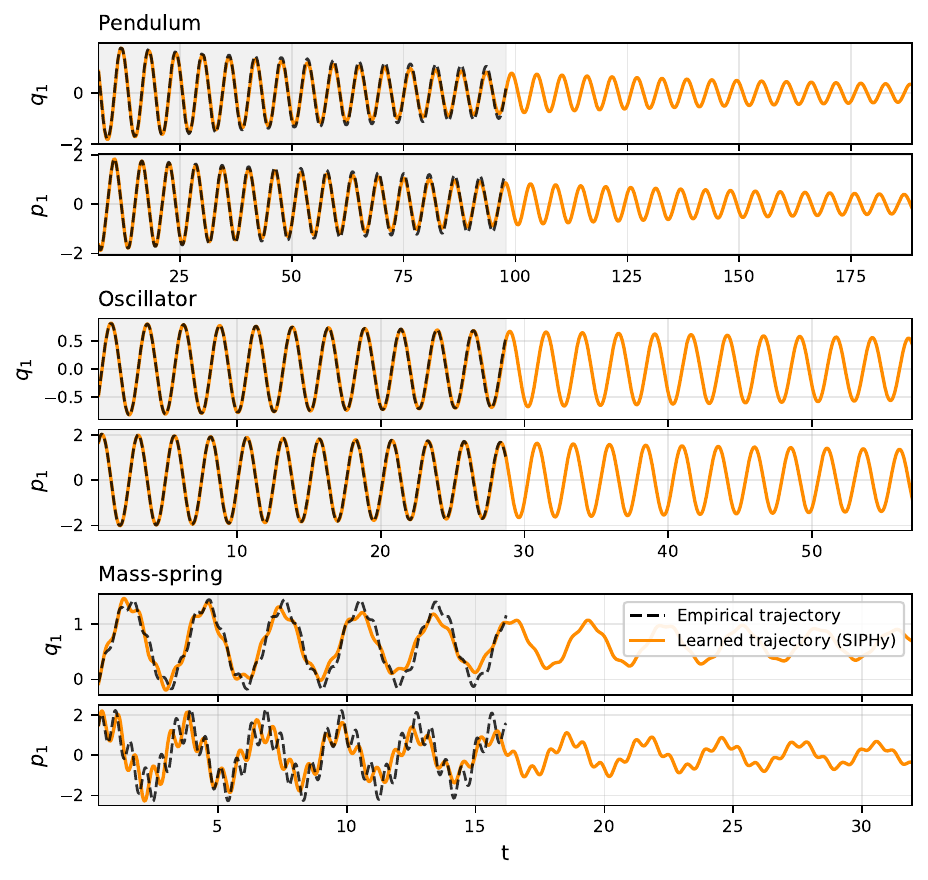}
\end{minipage}

\caption{Left: derivative phase portraits for the pendulum and oscillator, comparing finite-difference estimates (dashed gray) with Hamiltonian flow splines (orange). Right: empirical trajectories (dashed black) and SIPHy rollouts (orange) for all three systems. Only $(q_1,p_1)$ displayed for the mass-spring system. The gray region marks the identification window; the unshaded region shows extrapolation.}
\label{fig:real_rollouts}
\end{figure}

\begin{table}[!t]
  \caption{Port-Hamiltonian systems identified from experimental measurements
\cite{schmidt2009distilling}. Identified 
Hamiltonian $H(\mathbf{y})$, dissipation matrix $\mathbf{R}$, expressed in  scaled
coordinates, together with the rollout error.}
\label{tab:real_data}
\centering
\renewcommand{\arraystretch}{1.25}
\setlength{\tabcolsep}{3pt}
\resizebox{\textwidth}{!}{
\begin{tabular}{llll}
\toprule
System & $H(\mathbf{y})$ & $\mathbf{R}$ & NMSE \\
\midrule
Pendulum & $0.269 y_{0}^{2} + 0.519 y_{1}^{2} - 0.716 \cos{\left(y_{0} \right)}$ & $\left[\begin{matrix}0.017\end{matrix}\right]$ & 0.0121 \\
Oscillator & $3.056 y_{0}^{2} + 0.508 y_{1}^{2}$ & $\left[\begin{matrix}0.014\end{matrix}\right]$ & 0.0027 \\
Mass-spring & $\begin{aligned}&18.31 y_{0}^{2} - 30.437 y_{0} y_{1} - 0.279 y_{0} y_{2} - 42.157 y_{0} + 16.633 y_{1}^{2} + 0.184 y_{1} y_{2} \\ &+ 39.994 y_{1} + 0.61 y_{2}^{2} - 0.172 y_{2} y_{3} + 0.29 y_{2} + 0.566 y_{3}^{2}\end{aligned}$ & $\left[\begin{matrix}0.156 & 0\\0 & 0\end{matrix}\right]$ & 0.2754 \\
\bottomrule
\end{tabular}
}
\end{table}

\subsection{Identification from experimental data}
\label{subsec:real_data}

We finally apply SIPHy to physical measurements from \cite{schmidt2009distilling}, comprising a pendulum, a linear oscillator and a two-mass spring system. Unlike the synthetic experiments, the true Hamiltonian and dissipation are unknown, the data are slightly noisy and the sampling is mildly irregular. We adapt the pipeline in three minor ways, while keeping SIPHy and the flow splines otherwise unchanged. We apply a single scaling factor $\mathbf{y} \mapsto \mathbf{y}/s$ with $s$ being the empirical standard deviation of the observed states. A uniform scalar preserves the canonical $(\mathbf{q},\mathbf{p})$ pairing, so the identified system remains a valid port-Hamiltonian system, expressed in scaled coordinates. We use fixed hyperparameters in the flow splines and SIPHy algorithm and these are reported in \cref{sec:details_num}. 

Physical damping in these systems is small, and it becomes challenging to balance Hamiltonian sparsity against the identification of weak dissipation. We therefore threshold the coefficient vector $\mathbf{w}=[\mathbf{c}\;\mathbf{r}]$ block-wise, using $\tau$ on the Hamiltonian weights $\mathbf{c}$ and a smaller threshold $\tau_R =10^{-4}$ on the dissipation weights $\mathbf{r}$. Additionally, we refine the identified dissipation $\mathbf{R} = \mathrm{diag}(\mathbf{r})$ by minimizing the trajectory-rollout error
against the smoothed flow spline $\mathbf{z}(t)$, that is, by solving
\begin{equation*}
  \min_{\mathbf{R} = \mathrm{diag}(\mathbf{r}),\; \mathbf{r} \ge 0}
    \; e\big(\mathbf{y}_w(t),\, \mathbf{z}(t)\big).
    \label{eq:rollout_refine}
\end{equation*}
Here, $\mathbf{y}_w$ is the trajectory
obtained by integrating the identified system from $\mathbf{z}(0)$. The non-negativity
constraint guarantees an identified model that satisfies the port-Hamiltonian power-balance. Derivative phase portraits as well as empirical and learned trajectories are shown in \cref{fig:real_rollouts}. The identified systems are found in \cref{tab:real_data}.

\subsection{Implementation}

The implementation of SIPHy leverages PySINDy \cite{de2020pysindy} with a port-Hamiltonian feature library. The bound-constrained refit \eqref{eq:refit} is solved with SciPy \cite{virtanen2020scipy} least-squares with bounds applied to the Cholesky factor of the
regularized normal equations $\boldsymbol{\Theta}_S^\top\boldsymbol{\Theta}_S + \alpha \mathbf{I}$, which reduces the problem to $|S|$ rows. The optimization problems in \eqref{eq:optimal_control} and \eqref{eq:optimal_control_k} are solved with Sequential Least Squares Programming (SLSQP) implemented in SciPy \cite{virtanen2020scipy}, where gradients are computed using ODE adjoints from Diffrax \cite{kidger2022neural}. Optuna \cite{akiba2019optunanextgenerationhyperparameteroptimization} is used for hyperparameter optimization, as described in greater detail in \cref{sec:hyper_selection}.

As a benchmark, we apply SINDy \cite{brunton_discovering_2016} using the STLSQ optimizer implemented in PySINDy \cite{de2020pysindy} with the same feature library and setup for hyperparameter search as SIPHy, enabling a direct comparison. Derivatives are estimated by preprocessing the noisy trajectory with \texttt{polydiff} from PyNumDiff \cite{van_breugel_numerical_2020}.

\section{Conclusion}

We have introduced Hamiltonian flow splines and demonstrated how this algorithm allows us to approximate smooth trajectories and derivatives from noisy observations of port-Hamiltonian systems. The SIPHy algorithm is derived under the assumption that the dissipation and control only act on the conjugate momentum $\mathbf{p}$. Thus, we are able to write a subclass of port-Hamiltonian systems as a linear combination of a gradient basis, the derivative of the position $\dot{\mathbf{q}}$ and the control $\mathbf{u}$. The numerical experiments demonstrate how this framework often identifies the correct basis terms, even with high levels of noise and missing data from a single trajectory. This is even true for the mass--spring chain with $d=5$ masses, which includes friction and control. Identifying the port-Hamiltonian structure provides a principled approach for simulating the system in other scenarios, such as zero dissipation or different control inputs.

One limitation of the flow splines is the combinatorial growth of basis terms $N_\Lambda$ with increasing state dimension $n$. The current algorithm uses uniformly distributed knots $T_k$, i.e. points in time where the flow $\mathbf{z}(t)$ switches from one monomial coefficient vector $\mathbf{a}^{(k)}$ to the next $\mathbf{a}^{(k+1)}$. Choosing the location of the knots in an adaptive manner could potentially offer improved accuracy. Moreover, the presented SIPHy algorithm is only tractable for systems with the canonical skew-symmetric structure matrix $\mathbf{J}$. Many port-Hamiltonian systems, such as PH-systems on graphs, have a formulation determined by the connectivity structure of the physical system, which offers a natural venue for generalizing the current approach. 

The main advantage of the port-Hamiltonian representation is that the learned Hamiltonian, dissipation, and input terms can be inspected and modified separately. In the numerical experiments, this decomposition allows the identified models to be evaluated after changing the dissipation or control input. Future work could investigate whether the same structure is useful for downstream tasks such as passivity-based control or structure-preserving model reduction.

\appendix

\section{Hyperparameter selection}\label{sec:hyper_selection}

For the Hamiltonian flow splines introduced in \cref{sec:energy_splines}, we use the restricted polynomial basis with maximum degree $p=6$ for all experiments except the mass-spring chain where $p=4$ is used to reduce the computational complexity. Additionally, we need to select hyperparameters $\theta =\{M, r, \lambda\}$. The smoothing algorithm optimizes the flow over $M\cdot r$ points in time, but only keeps the first $M$ points in the trajectory where $r$ is a window-length multiplier. $\lambda$ is the $l_2$-regularization parameter. We find that, for data with step size $h=0.01$ in a time interval $t\in [0,20]$, the following parameter ranges work well
\begin{align*}
   \left\{ 40 \leq M \leq 60,\quad r=3,  \quad  10^{-5} \leq\lambda \leq 10^{-4} \right\}.
\end{align*}

For SIPHy, we use a monomial feature library of degree $p = 5$ for all systems except the mass-spring chain, where $p = 3$ is used. SIPHy has hyperparameters $\phi = \{\alpha,\tau,\beta,\texttt{TRIG}\}$. $\alpha$ is the $l_2$-regularization parameter, \(\tau\) is a threshold for sparsity, \(\beta \in [0,0.1]\) is the fraction of time points at the beginning and end of the time interval $[0,T]$ held out during STLSQ to avoid boundary effects of the preprocessing algorithm, and $\texttt{TRIG} \in \{\texttt{True}, \texttt{False} \}$ is a boolean variable to select whether to include trigonometric functions in the Hamiltonian basis $\boldsymbol{\Xi}$. We perform hyperparameter optimization over the set
\begin{align*}
 \left\{ 0.1 \leq \alpha \leq 0.4, \;\; 0.05 \leq \tau \leq 0.2, \;\; 0 \leq  \beta \leq 0.1   \right \}.
\end{align*}
We write
\[
\mathbf{w} = \texttt{SIPHy}(\mathbf{z},\dot{\mathbf{z}},\phi)
\]
for the identified port-Hamiltonian system with coefficients \(\mathbf{w}\) from the preprocessed trajectory \(\mathbf{z},\dot{\mathbf{z}}\) with hyperparameters \(\phi\) by solving the optimization problem in \eqref{eq:phs_sindy}. Integrating the identified system $\dot{\mathbf{y}} = \boldsymbol{\Theta}(\mathbf{y}) \mathbf{w}$ gives the trajectory  \(\mathbf{y}_w(t)\) for the given hyperparameter set \(\phi\). For a given choice of SIPHy hyperparameters $\phi$, we evaluate the identified model $\mathbf{w}$ with the residual sum of squares including a $l_0$ (number of non-zero coefficients) complexity penalty with decreasing weight for increasing system dimensionality $n$
\[
S(\mathbf{w}) = \log \left ( \sum_{i=1}^{N} \big\| \mathbf{y}_w(t_i) - \tilde{\mathbf{y}}(t_i) \big\|_2^2 \right ) + \frac{1}{n}\|\mathbf{w}\|_0 
\]
which is inspired by the Akaike information criterion (AIC) applied for sparse regression in \cite{mangan2017model}.

\section{Details of numerical experiments}\label{sec:details_num}

\begin{itemize}

  \item \textbf{Train/validation split.}
    After edge-cropping by $\beta \in [0,0.1]$, the first $80\%$ of the remaining time points are used for STLSQ fitting. The selection metric $S(\mathbf{w})$ is computed over the full trajectory, excluding the cropped initial and last points in time.

  \item \textbf{Optuna budget and warm-up.}
    Each hyperparameter search runs $50$ TPE trials, of which the first $\lfloor 50/4 \rfloor = 12$ are drawn uniformly at random before the surrogate model is engaged.

  \item \textbf{Fixed smoother parameters.}
    The smoother parameters used are found in \cref{tab:smoother_params}.
    
    \begin{table}[h!]
    \caption{Flow spline parameter settings for the three experiments.}
\label{tab:smoother_params}
\renewcommand{\arraystretch}{1.3}
\centering
\begin{tabular}{llll}
\hline
 & $M$ & $r$ & $\lambda$ \\
\hline
Mass-spring & $50$ & $3$ & $10^{-4}$ \\
Dissipative $c=0.1$ & $60$ & $3$ & $10^{-4}$ \\
Conservative $c=0$ & $40$ & $3$ & $10^{-5}$ \\
\hline
\end{tabular}
\end{table}

  \item \textbf{Continuity order.}
    The spline continuity constraint at knot points is set to $\mathcal{C}^2$
    for 2-dimensional systems (Pendulum, Cubic) and $\mathcal{C}^3$ for Hénon--Heiles, FPUT and the mass-spring chain.

  \item \textbf{Initial conditions.}
    Training and test initial conditions for the different systems are found in \cref{tab:inits} and \cref{tab:train-test-settings}.
\begin{table}[ht]
    \caption{Initial values used in the numerical experiments}
    \label{tab:inits}
\centering
\renewcommand{\arraystretch}{1.3}
    \begin{tabular}{lll}
      \toprule
      System & $\mathbf{y}_0$ (train) & $\mathbf{y}_0$ (test) \\
      \midrule
      Pendulum       & $[0.9,\;1.3]$                        & $[-0.1,\;0.2]$ \\
      Cubic          & $[1.2,\;-0.3]$                       & $[0.8,\;-0.7]$ \\
      Hénon--Heiles  & $[-0.2,\;0.15,\;-0.15,\;0.2]$       & $[-0.1,\;0.2,\;-0.3,\;0.25]$ \\
      FPUT           & $[-0.2,\;0.15,\;0.1,\;0.25]$         & $[0.3,\;0.3,\;0.2,\;0.2]$ \\
      \bottomrule
    \end{tabular}
    \end{table}

\end{itemize}

\begin{table}[!h]
\caption{Configurations for the mass--spring chain experiment.}
\label{tab:train-test-settings}
\centering
\begin{tabular}{lll}
\toprule
 &Training & Test\\
\midrule
$u(t)$ &
$\exp(-0.5t)\cos(t)$ &
$0.1\sin(t)$ \\

$\mathbf{R}$ &
$0.15\cdot \mathbf{I}_d$ &
$0\cdot \mathbf{I}_d$ \\

$\mathbf{q}_0$ &
$[0.20,\;0.35,\;0.40,\;0.35,\;0.20]$ &
$[0.17,\;0.17,\;0.0,\;-0.17,\;-0.17]$ \\

$\mathbf{p}_0$ &
$[0.10,\;0.05,\;-0.01,\;-0.06,\;-0.11]$ &
$\mathbf{0}$ \\
\bottomrule
\end{tabular}
\end{table}

\begin{table}[!h]
\caption{Fixed hyperparameters for the experimental-data identification. The
flow-spline parameters $p$ (degree), $M$ (window points kept), $r$
(window-length multiplier), $\lambda$ ($\ell_2$ regularization) and continuity
order, and the SIPHy parameters (Hamiltonian degree, \texttt{TRIG}, $\ell_2$
regularization $\alpha$, Hamiltonian-block threshold $\tau$ and
dissipation/input-block threshold $\tau_R$) follow the notation of
\cref{sec:hyper_selection}.}
\label{tab:real_hyperparams}
\centering
\renewcommand{\arraystretch}{1.3}
\setlength{\tabcolsep}{5pt}
\begin{tabular}{lccccc@{\hskip 1.5em}ccccc}
\toprule
& \multicolumn{5}{c}{\textbf{Flow splines}} & \multicolumn{5}{c}{\textbf{SIPHy}} \\
\cmidrule(lr){2-6}\cmidrule(lr){7-11}
System & $p$ & $M$ & $r$ & $\lambda$ & cont. & deg. & \texttt{TRIG} & $\alpha$ & $\tau$ & $\tau_R$ \\
\midrule
Pendulum    & $5$ & $10$ & $2$ & $10^{-3}$ & $\mathcal C^3$ & $2$ & \checkmark & $0.1$ & $0.05$ & $10^{-4}$ \\
Oscillator  & $5$ & $10$ & $2$ & $10^{-3}$ & $\mathcal C^3$ & $2$ & $\times$   & $0.1$ & $0.05$ & $10^{-4}$ \\
Mass--spring& $5$ & $10$ & $2$ & $10^{-3}$ & $\mathcal C^3$ & $2$ & $\times$   & $0.1$ & $0.05$ & $10^{-4}$ \\
\bottomrule
\end{tabular}
\end{table}

\section{Noise levels and error metric}\label{sec:noise_level}

Let $\mathbf{y}(t_n) = \varphi_{t,f}(\mathbf{y}_0)$ be the flow of a port-Hamiltonian system with vector field $\mathbf{f}$ as in \eqref{eq:port_hamiltonian} from an initial condition $\mathbf{y}_0\in \mathbb{R}^n$. The observed trajectory with noise-to-signal ratio $\gamma \in [0,1]$ is given by  
\[
\tilde{\mathbf{y}}_i = \mathbf{y}(t_i) + \boldsymbol{\delta}_i,\qquad \boldsymbol{\delta}_i \sim \mathcal N\left (0,\frac{\gamma^2\sigma_y^2}{n}\mathbf{I}\right).
\]
where $\sigma_y^2$ is the total variance, or the trace of the empirical covariance matrix, given by
\begin{align}
\sigma_y^2 = \mathrm{tr}(\boldsymbol{\Sigma}_y)  =\frac{1}{N} \sum_{i=1}^N\|\mathbf{y}_i-\overline{\mathbf{y}}\|_2^2, \qquad \overline{\mathbf{y}} &=\frac{1}{N} \sum_{i=1}^N \mathbf{y}_i.
\label{eq:total_variance}
\end{align} This definition of relative noise is invariant to rotations and translations of $\mathbf{y}_i$ and ensures that $\gamma$ is a noise-to-signal ratio in the sense of the square root of the total variance, since we have that
$
\gamma  =  \frac{\sigma_{\delta}}{\sigma_y}
$.

\section{Additional experimental results}\label{sec:add_results}

\begin{table}[!h]
  \caption{Identified port-Hamiltonian systems with varying levels of noise $\gamma$ for the mass-spring system. }
  \label{tab:mass_spring_identified}
\centering
\renewcommand{\arraystretch}{1.25}
\setlength{\tabcolsep}{3pt}
\resizebox{\textwidth}{!}{
\begin{tabular}{ll}
\toprule
$ $ & \textbf{True Mass-Spring-Damper} \\
\midrule
& $H(\mathbf{y}) = 0.4 y_{0}^{2} - 0.4 y_{0} y_{1} + 0.4 y_{1}^{2} - 0.4 y_{1} y_{2} + 0.4 y_{2}^{2} - 0.4 y_{2} y_{3} + 0.4 y_{3}^{2} - 0.4 y_{3} y_{4} + 0.2 y_{4}^{2} + 0.5 y_{5}^{2} + 0.5 y_{6}^{2} + 0.5 y_{7}^{2} + 0.5 y_{8}^{2} + 0.5 y_{9}^{2}$ \\
& $\mathrm{diag}(\mathbf{R})=[0.15, 0.15, 0.15, 0.15, 0.15] \qquad \mathrm{diag}(\mathbf{G})=[0.0, 0.0, 0.0, 0.0, 1.0]$ \\
\midrule
$\gamma \;\;$ & \textbf{Identified port-Hamiltonian system}  \\
\midrule
0.0 & $H(\mathbf{y})=0.3997 y_{0}^{2} - 0.3998 y_{0} y_{1} + 0.3997 y_{1}^{2} - 0.3998 y_{1} y_{2} + 0.3996 y_{2}^{2} - 0.3996 y_{2} y_{3} + 0.3999 y_{3}^{2} - 0.4003 y_{3} y_{4} + 0.2004 y_{4}^{2} + 0.5 y_{5}^{2} + 0.5 y_{6}^{2} + 0.5 y_{7}^{2} + 0.5 y_{8}^{2} + 0.5005 y_{9}^{2}$ \\
& $\mathrm{diag}(\mathbf{R})=[0.1462, 0.1459, 0.1459, 0.1465, 0.1511] \qquad \mathrm{diag}(\mathbf{G})=[0.0, 0.0, 0.0, 0.0, 0.9987]$ \\
\hline
0.1 & $H(\mathbf{y})=0.4033 y_{0}^{2} - 0.4033 y_{0} y_{1} + 0.4016 y_{1}^{2} - 0.4027 y_{1} y_{2} + 0.3982 y_{2}^{2} - 0.3966 y_{2} y_{3} + 0.4019 y_{3}^{2} - 0.4037 y_{3} y_{4} + 0.2014 y_{4}^{2} + 0.4952 y_{5}^{2} + 0.4992 y_{6}^{2} + 0.5017 y_{7}^{2} + 0.4974 y_{8}^{2} + 0.4984 y_{9}^{2}$ \\
& $\mathrm{diag}(\mathbf{R})=[0.1335, 0.1556, 0.1496, 0.1563, 0.1371] \qquad \mathrm{diag}(\mathbf{G})=[0.0, 0.0, 0.0, 0.0, 0.9914]$ \\
\hline
0.2 & $H(\mathbf{y})=0.4045 y_{0}^{2} - 0.4002 y_{0} y_{1} + 0.1217 y_{0} y_{5}^{2} - 0.097 y_{0} y_{7}^{2} + 0.4005 y_{1}^{2} - 0.4018 y_{1} y_{2} + 0.3982 y_{2}^{2} - 0.3947 y_{2} y_{3} + 0.3962 y_{3}^{2} - 0.3957 y_{3} y_{4} + 0.1961 y_{4}^{2} + 0.4842 y_{5}^{2} + 0.4967 y_{6}^{2} + 0.5134 y_{7}^{2}$  \\
&\qquad $+ 0.5025 y_{8}^{2} + 0.489 y_{9}^{2}$\\
& $\mathrm{diag}(\mathbf{R})=[0.1345, 0.1659, 0.1557, 0.1417, 0.1457] \qquad \mathrm{diag}(\mathbf{G})=[0.0, 0.0, 0.0, 0.0, 1.0625]$ \\
\bottomrule
\end{tabular}
}
\end{table}

\begin{table}[!h]
  \label{tab:coefficient_table}
\caption{Identified port-Hamiltonian systems with varying levels of noise $\gamma$ and dissipation rates $c$ where the dissipation matrix is given by $\mathbf{R}=c\cdot \mathbf{I}_d$. }
\centering
\renewcommand{\arraystretch}{1.25}
\setlength{\tabcolsep}{3pt}
\resizebox{\textwidth}{!}{
\begin{tabular}{llll}
\toprule
$\gamma \;\;$ & $c$ & \textbf{Identified Hamiltonian} & $\mathrm{diag}(\mathbf{R})$  \\
  &  & \textbf{Pendulum}: \quad $H(\mathbf{y}) = 0.5 y_{1}^{2} - \cos{\left(y_{0} \right)}$ &   \\
\midrule
 0.0 & 0.1 & $0.5 y_{1}^{2} - 1.0 \cos{\left(y_{0} \right)}$ & $[0.1]$  \\
 0.1 & 0.1 & $0.495 y_{1}^{2} - 1.008 \cos{\left(y_{0} \right)}$ & $[0.098]$  \\
 0.2 & 0.1 & $0.486 y_{1}^{2} - 1.017 \cos{\left(y_{0} \right)}$ & $[0.102]$  \\
 \midrule
 0.0 & 0.0 & $0.5 y_{1}^{2} - 1.0 \cos{\left(y_{0} \right)}$ & $[0.0]$  \\
 0.1 & 0.0 & $0.494 y_{1}^{2} - 1.003 \cos{\left(y_{0} \right)}$ & $[0.0]$  \\
 0.2 & 0.0 & $0.486 y_{1}^{2} - 1.009 \cos{\left(y_{0} \right)}$ & $[0.0]$  \\
\midrule
  &  & \textbf{Cubic oscillator}: \quad $H(\mathbf{y}) = 0.25 y_{0}^{4} + 0.25 y_{1}^{4}$ &   \\
\midrule
 0.0 & 0.1 & $0.25 y_{0}^{4} + 0.25 y_{1}^{4}$ & $[0.1]$  \\
 0.1 & 0.1 & $0.245 y_{0}^{4} + 0.251 y_{1}^{4}$ & $[0.097]$  \\
 0.2 & 0.1 & $0.239 y_{0}^{4} + 0.248 y_{1}^{4}$ & $[0.093]$  \\
 \midrule
 0.0 & 0.0 & $0.25 y_{0}^{4} + 0.25 y_{1}^{4}$ & $[0.0]$  \\
 0.1 & 0.0 & $0.25 y_{0}^{4} + 0.251 y_{1}^{4}$ & $[0.0]$  \\
 0.2 & 0.0 & $0.25 y_{0}^{4} + 0.25 y_{1}^{4}$ & $[0.0]$  \\
\midrule
  &  & \textbf{H\'enon--Heiles}: \quad $H(\mathbf{y}) = y_{0}^{2} y_{1} + 0.5 y_{0}^{2} - 0.333 y_{1}^{3} + 0.5 y_{1}^{2} + 0.5 y_{2}^{2} + 0.5 y_{3}^{2}$ &   \\
\midrule
 0.0 & 0.1 & $1.0 y_{0}^{2} y_{1} + 0.5 y_{0}^{2} - 0.333 y_{1}^{3} + 0.5 y_{1}^{2} + 0.5 y_{2}^{2} + 0.5 y_{3}^{2}$ & $[0.1, 0.1]$  \\
 0.1 & 0.1 & $0.064 y_{0}^{3} + 1.116 y_{0}^{2} y_{1} + 0.494 y_{0}^{2} - 0.346 y_{1}^{3} + 0.501 y_{1}^{2} + 0.504 y_{2}^{2} + 0.5 y_{3}^{2}$ & $[0.09, 0.105]$  \\
 0.2 & 0.1 & $1.217 y_{0}^{2} y_{1} + 0.498 y_{0}^{2} + 0.336 y_{0} y_{1}^{2} - 0.251 y_{1}^{3} + 0.494 y_{1}^{2} + 0.504 y_{2}^{2} + 0.501 y_{3}^{2}$ & $[0.0, 0.114]$  \\
 \midrule
 0.0 & 0.0 & $1.0 y_{0}^{2} y_{1} + 0.5 y_{0}^{2} - 0.333 y_{1}^{3} + 0.5 y_{1}^{2} + 0.5 y_{2}^{2} + 0.5 y_{3}^{2}$ & $[0.0, 0.0]$  \\
 0.1 & 0.0 & $0.974 y_{0}^{2} y_{1} + 0.494 y_{0}^{2} - 0.337 y_{1}^{3} + 0.504 y_{1}^{2} + 0.505 y_{2}^{2} + 0.493 y_{3}^{2}$ & $[0.0, 0.0]$  \\
 0.2 & 0.0 & $0.97 y_{0}^{2} y_{1} + 0.491 y_{0}^{2} - 0.346 y_{1}^{3} + 0.502 y_{1}^{2} + 0.507 y_{2}^{2} + 0.497 y_{3}^{2}$ & $[0.0, 0.0]$  \\
\midrule
  &  & \textbf{FPUT}: \quad $H(\mathbf{y}) = 0.5 y_{0}^{4} + 3.0 y_{0}^{2} y_{1}^{2} + 0.5 y_{1}^{4} + 0.5 y_{1}^{2} + 0.5 y_{2}^{2} + 0.5 y_{3}^{2}$ &   \\
\midrule
 0.0 & 0.1 & $0.5 y_{0}^{4} + 3.0 y_{0}^{2} y_{1}^{2} + 0.5 y_{1}^{4} + 0.5 y_{1}^{2} + 0.5 y_{2}^{2} + 0.5 y_{3}^{2}$ & $[0.1, 0.1]$  \\
 0.1 & 0.1 & $0.45 y_{0}^{4} + 3.439 y_{0}^{2} y_{1}^{2} - 0.094 y_{0}^{2} y_{3}^{2} + 0.12 y_{0} y_{1}^{3} + 0.609 y_{1}^{4} + 0.483 y_{1}^{2} + 0.498 y_{2}^{2} + 0.499 y_{3}^{2}$ & $[0.098, 0.1]$  \\
 0.2 & 0.1 & $4.382 y_{0}^{2} y_{1}^{2} + 0.482 y_{1}^{2} + 0.494 y_{2}^{2} + 0.499 y_{3}^{2}$ & $[0.0, 0.0]$  \\
 \midrule
 0.0 & 0.0 & $0.5 y_{0}^{4} + 3.0 y_{0}^{2} y_{1}^{2} + 0.5 y_{1}^{4} + 0.5 y_{1}^{2} + 0.5 y_{2}^{2} + 0.5 y_{3}^{2}$ & $[0.0, 0.0]$  \\
 0.1 & 0.0 & $0.5 y_{0}^{4} + 2.958 y_{0}^{2} y_{1}^{2} + 0.746 y_{1}^{4} + 0.476 y_{1}^{2} + 0.5 y_{2}^{2} + 0.502 y_{3}^{2}$ & $[0.0, 0.0]$  \\
 0.2 & 0.0 & $3.517 y_{0}^{2} y_{1}^{2} - 0.289 y_{0}^{2} y_{1} y_{3} + 0.893 y_{0}^{2} y_{3}^{2} + 0.521 y_{1}^{2} + 0.497 y_{2}^{2} + 0.454 y_{3}^{2}$ & $[0.0, 0.0]$  \\
\bottomrule
\end{tabular}
}
\end{table}

\begin{figure}[!h]
\centering
\includegraphics[clip, trim=0 1.4cm 0 0, width=1\textwidth]{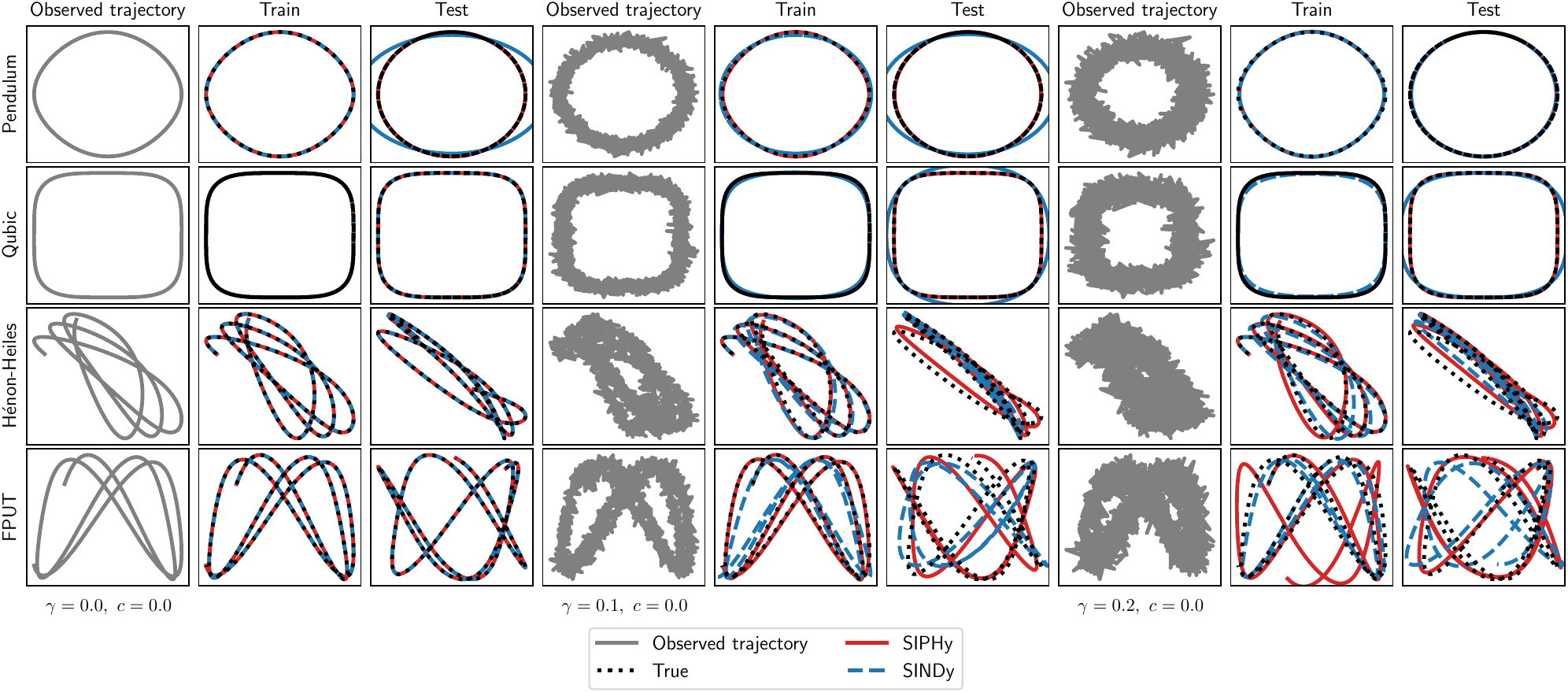}
\includegraphics[clip, trim=0 0 0 0.5cm, width=1\textwidth]{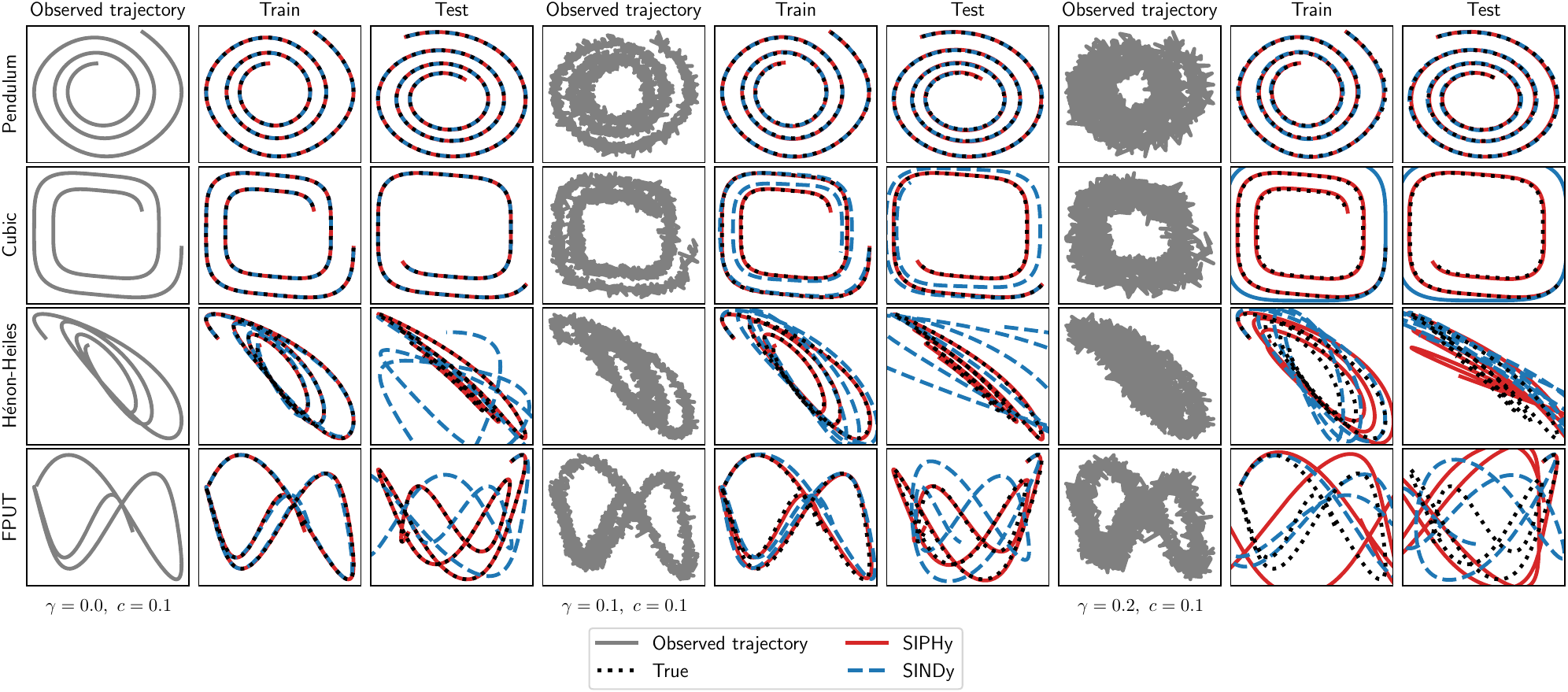}
\caption{Observed and learned trajectories from training and test initial values $\mathbf{y}_0$ for the proposed SIPHy algorithm and SINDy. Increasing noise levels from left to right. Trajectories in the top block are conservative while the lower block the dissipation rate is given by $c=0.1$.}
\label{fig:full_plot}
\end{figure}

\FloatBarrier

\bibliographystyle{abbrvnat}
\bibliography{ref}

\end{document}

%% file: tables/system_table_compare_0.0_0.1_final.tex
\begin{tabular}{l c|cc|cc|}
\toprule
\textbf{System} & \textbf{Noise} & \multicolumn{2}{c}{Coefficient Error} & \multicolumn{2}{c}{Term Match} \\
$c=0$  &  & SIPHy & SINDy & SIPHy & SINDy \\
\midrule
\multirow{3}{*}{\parbox{0.4in}{\centering \includegraphics[height=0.3in]{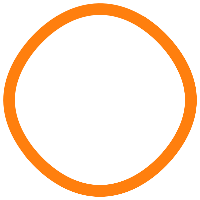} \\ \small Pendulum}} & 0.0 & \textbf{1.1e-15} & 5.8e-01 & $\checkmark$ & $\times$ \\
 & 0.1 & \textbf{6.4e-03} & 6.4e-01 & $\checkmark$ & $\times$ \\
 & 0.2 & \textbf{1.7e-02} & 6.0e-02 & $\checkmark$ & $\checkmark$ \\
\midrule
\multirow{3}{*}{\parbox{0.4in}{\centering \includegraphics[height=0.3in]{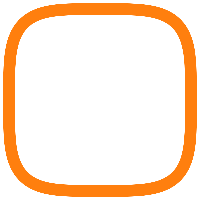} \\ \small Cubic}} & 0.0 & \textbf{2.2e-16} & 1.4e-12 & $\checkmark$ & $\checkmark$ \\
 & 0.1 & \textbf{5.5e-04} & 2.2e+00 & $\checkmark$ & $\times$ \\
 & 0.2 & \textbf{6.2e-04} & 1.4e+00 & $\checkmark$ & $\times$ \\
\midrule
\multirow{3}{*}{\parbox{0.4in}{\centering \includegraphics[height=0.3in]{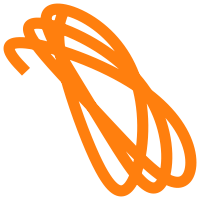} \\ \small Hénon}} & 0.0 & \textbf{1.2e-15} & 2.3e-15 & $\checkmark$ & $\checkmark$ \\
 & 0.1 & \textbf{2.9e-02} & 2.7e-01 & $\checkmark$ & $\checkmark$ \\
 & 0.2 & \textbf{3.4e-02} & 3.2e-01 & $\checkmark$ & $\checkmark$ \\
\midrule
\multirow{3}{*}{\parbox{0.4in}{\centering \includegraphics[height=0.3in]{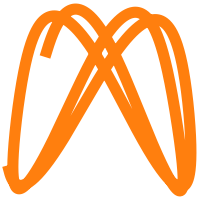} \\ \small FPUT}} & 0.0 & \textbf{1.7e-14} & 1.3e-12 & $\checkmark$ & $\checkmark$ \\
 & 0.1 & \textbf{2.5e-01} & 7.5e+00 & $\checkmark$ & $\times$ \\
 & 0.2 & \textbf{1.3e+00} & 1.7e+01 & $\times$ & $\times$ \\
\bottomrule
\end{tabular}
\begin{tabular}{l c|cc|cc}
\toprule
\textbf{System} & \textbf{Noise} & \multicolumn{2}{c}{Coefficient Error} & \multicolumn{2}{c}{Term Match} \\
 $c=0.1$ &  & SIPHy & SINDy & SIPHy & SINDy \\
\midrule
\multirow{3}{*}{\parbox{0.4in}{\centering \includegraphics[height=0.3in]{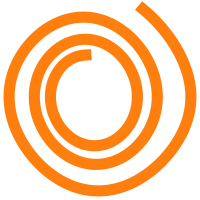} \\ \small Pendulum}} & 0.0 & \textbf{6.8e-16} & 9.9e-01 & $\checkmark$ & $\times$ \\
 & 0.1 & \textbf{5.9e-03} & 9.7e-01 & $\checkmark$ & $\times$ \\
 & 0.2 & 2.4e-01 & \textbf{2.0e-01} & $\times$ & $\times$ \\
\midrule
\multirow{3}{*}{\parbox{0.4in}{\centering \includegraphics[height=0.3in]{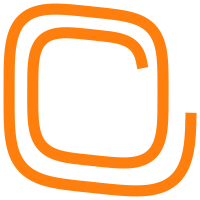} \\ \small Cubic}} & 0.0 & \textbf{2.6e-16} & 8.3e-16 & $\checkmark$ & $\checkmark$ \\
 & 0.1 & \textbf{7.6e-03} & 4.3e-01 & $\checkmark$ & $\times$ \\
 & 0.2 & \textbf{1.7e-02} & 1.5e-01 & $\checkmark$ & $\times$ \\
\midrule
\multirow{3}{*}{\parbox{0.4in}{\centering \includegraphics[height=0.3in]{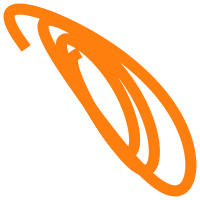} \\ \small Hénon}} & 0.0 & \textbf{3.9e-14} & 2.6e+00 & $\checkmark$ & $\times$ \\
 & 0.1 & \textbf{3.8e-02} & 1.6e+00 & $\checkmark$ & $\times$ \\
 & 0.2 & \textbf{1.1e+00} & 2.6e+00 & $\times$ & $\times$ \\
\midrule
\multirow{3}{*}{\parbox{0.4in}{\centering \includegraphics[height=0.3in]{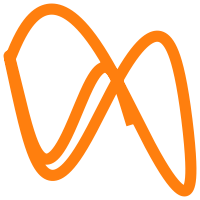} \\ \small FPUT}} & 0.0 & \textbf{1.5e-11} & 2.2e+00 & $\checkmark$ & $\times$ \\
 & 0.1 & 7.9e+00 & \textbf{4.7e+00} & $\times$ & $\times$ \\
 & 0.2 & \textbf{1.6e+00} & 5.0e+00 & $\times$ & $\times$ \\
\bottomrule
\end{tabular}

%% file: mass_spring_damper_fig.tex
\begin{tikzpicture}
\centering
  \def\H{0.8}   
  \def\T{0.3}   
  \def\W{5.0}   
  \def\D{0.25}  
  \def\h{0.6}   
  \def\w{0.8}   

  \def\xA{1.0}
  \def\xB{2.5}
  \def\xC{4.0}
  \def\xD{6.2}

  \draw[spring, segment length=6.0]
    (0,\h/2) -- (\xA,\h/2)
     node[midway, above=0.2] {$k_1$};

  \draw[spring, segment length=5.2]
    (\xA+\w,\h/2) -- (\xB,\h/2)
     node[midway, above=0.2] {$k_2$};

  \draw[spring, segment length=5.2]
    (\xB+\w,\h/2) -- (\xC,\h/2)
     node[midway, above=0.2] {$k_3$};



  \draw[ground]
    (0,0) |-++ (-\T,\H) |-++ (\T+\W,-\H-\D)
    |-++ (-\W,\D) -- cycle;

  \draw (0,\H) -- (0,0) -- (\W,0);

  \draw[mass] (\xA,0) rectangle++ (\w,\h) node[midway] {$m_1$};
  \draw[mass] (\xB,0) rectangle++ (\w,\h) node[midway] {$m_2$};
  \draw[mass] (\xC,0) rectangle++ (\w,\h) node[midway] {$m_3$};

  \draw[force] (\xC+\w,0.5*\h) --++ (0.8,0)
   node[midway,above=0] {$u(t)$};

\end{tikzpicture}